\documentclass[12pt]{article}
\usepackage{authblk} % 管理多作者和单位
\usepackage{amssymb}
\usepackage{amsmath}
\usepackage{subfigure}
\numberwithin{equation}{section}%公式按章节编号
\numberwithin{figure}{section}%图片按章节编号
\numberwithin{table}{section}%表格按章节编号
\usepackage{stmaryrd} % 三条竖线的范数
\usepackage{float}%提供float浮动环境
\usepackage{tikz}
\usepackage{mathtools}
\usepackage{listings}
\usepackage{geometry}
\usepackage{graphicx}
\usepackage{xcolor}
\usepackage{float}
\usepackage{verbatim}
\usepackage{algorithmic}
\usepackage{algorithm}
\numberwithin{algorithm}{section}
\usepackage{enumerate}
\usepackage{theorem}
\usepackage{hyperref}

\numberwithin{equation}{section} % 公式按章节编号
\everymath{\displaystyle} % 全文的行内公式都显示行间公式效果
\usepackage[section]{placeins} % 禁止图片或者表格浮动超过\section的范围
\usepackage{soul, color, xcolor} % 设置文本高亮显示 \hl{}

\newtheorem{theorem}{Theorem}[section]

\newtheorem{problem}{Problem}[section]
\newtheorem{remark}{Remark}[section]

\newtheorem{proposition}{Proposition}[section]

\newtheorem{example}{Example}[section]
\newenvironment{proof}[1][Proof]{\textbf{#1.} }

\begin{document}

\title{Bioluminescence tomography via a shape optimization method based on a complex-valued model}

\author[1]{Qianqian Wu}
\author[1,2]{Rongfang Gong\thanks{Corresponding author. Emails: grf\_math@nuaa.edu.cn (R. Gong), wuqianqian@nuaa.edu.cn (Q. Wu), wgong@lsec.cc.ac.cn (W. Gong), zhangziyi21@mails.ucas.ac.cn (Z. Zhang), sfzhu@math.ecnu.edu.cn (S. Zhu)}} 
\author[3]{Wei Gong}
\author[4]{Ziyi Zhang}
\author[5]{Shengfeng Zhu}

\affil[1]{\small School of Mathematics, Nanjing University of Aeronautics and Astronautics, Nanjing, 211106, Jiangsu, China.}
\affil[2]{\small Key Laboratory of Mathematical Modelling and High Performance Computing of Air Vehicles (NUAA), MIIT, Nanjing, 211106, Jiangsu, China.}
\affil[3]{\small State Key Laboratory of Mathematical Sciences, Academy of Mathematics and Systems Science, Chinese Academy of Sciences, Beijing, 100190, China.}
\affil[4]{\small School of Mathematical Sciences,	University of Chinese Academy of Sciences, Beijing, 100049, China.}
\affil[5]{\small School of Mathematical Sciences \& Key Laboratory of Ministry of Education \& Shanghai Key Laboratory of Pure Mathematics and Mathematical Practice, East China Normal University, Shanghai 200241, China.}

\date{}
\maketitle
\vspace{-1em}
\abstract{In this study, we investigate the inverse source problem arising in bioluminescence tomography, the objective of which is to reconstruct both the support and the intensity of an internal light source from boundary measurements governed by an elliptic model. A shape optimization framework is developed in which the source intensity and its support are decoupled through first-order optimality conditions. To enhance the stability of the reconstruction, we incorporate a parameter-dependent coupled complex boundary method together with perimeter and volume regularizations. Source support is represented by a level set function, allowing the algorithm to naturally accommodate topological changes and recover multiple, closely spaced, or nested source regions. Theoretical justifications for the proposed formulation and regularization strategy are established, and extensive numerical experiments are performed to assess the reconstruction accuracy for both noise-free and noisy data. The results demonstrate that our method achieves robust and accurate recovery of source geometry and intensity, and exhibits clear advantages over existing approaches.}


\section{Introduction}\label{sec:introduction}
Bioluminescence tomography (BLT) is an emerging molecular imaging modality that has attracted considerable attention owing to its ability to noninvasively monitor physiological and pathological processes \textit{in vivo} at cellular and molecular levels (\cite{Bentley2022, WangBL2023}). This modality exploits the intrinsic light emitted by bioluminescent sources, thereby eliminating the need for external excitation, significantly reducing background noise, and enhancing imaging sensitivity. These advantages make BLT particularly well suited for small animal studies and preclinical applications (\cite{Deng2022, Shi2018}).

The central objective of BLT is to quantitatively reconstruct the spatial distribution and the intensity of internal bioluminescent sources from optical signals measured on a subject’s surface (\cite{Wang2006, Wang2008}). By solving this inverse problem, one can localize and characterize light-emitting regions within biological tissues. This imaging capability is particularly valuable for investigating tumor progression, gene expression, and other dynamic biological processes in real-time. Accordingly, BLT serves as a powerful tool for advancing our understanding of complex biological systems and evaluating the efficacy of therapeutic interventions.

In BLT research, diffusion approximation (DA) is commonly used to model photon propagation. Let $\Omega \subset \mathbb{R}^d $ ($d \le 3$) be an open bounded domain with a Lipschitz boundary $\Gamma := \partial\Omega$, and $\Gamma_0 \subset \Gamma$ denote the measured portion of the boundary. Without loss of generality, we formulate the steady-state BLT problem using the following elliptic boundary value problem (\cite{Han2008}):

\begin{problem}\label{prob:blt}
	Find a source function $q$ inside $\Omega$ such that the solution $u$ of the forward Robin boundary value problem (BVP)
	\begin{equation}
		\left \{
		\begin{aligned}
			- {\rm div} (D \nabla u) + \mu_{a} u &= q & \quad \textrm{in}\ \Omega,\\[0.5em]
			u + 2 A D \partial_n u &= g^{-} & \quad \textrm{on}\ \Gamma,
		\end{aligned}
		\right.
		\label{eq:bvp}
	\end{equation}
	satisfies the outgoing flux density on the boundary:
	\begin{equation}
		g = -D \partial_n u \quad \textrm{on}\ \Gamma_{0}.
		\label{eq:bvp_meas}
	\end{equation}
\end{problem}

Here, $D = 1/[3(\mu_a + \mu_s^{\prime})]$ denotes the diffusion coefficient, and $\mu_a$ and $\mu_s^{\prime}$ denote the absorption and reduced scattering coefficients, respectively; $\partial_n$ denotes the outward normal derivative and $g^{-}$ denotes the incoming flux on $\Gamma$, which vanishes when imaging is performed in a dark environment; $A=A(x)=(1+R(x))/(1-R(x))$, where $R(x) \approx -1.4399 \gamma(x)^{-2} + 0.7099 \gamma(x)^{-1} + 0.6681 + 0.0636 \gamma(x)$ and $\gamma(x)$ is the refractive index at $x\in\Gamma$. In the following, we consider the case $g^{-}\equiv0$ and $\Gamma_0=\Gamma$.

Problem \ref{prob:blt} does not ensure a unique reconstruction of the source and is essentially ill-posed. Numerical experiments indicate that with a single set of boundary measurements, a strong source confined to a small region may be indistinguishable from a weak source distributed over a larger region (\cite{Gong2016ccbm}). Such non-uniqueness poses a fundamental challenge in BLT, particularly in biomedical applications where precise localization is crucial. Consequently, obtaining a unique reconstruction of the source region from limited boundary data is a central and challenging problem in BLT research. To address this issue, existing studies have adopted three principal strategies.

The first strategy is to enrich the measurement data used during reconstruction. For example, \cite{Gong2016rte, GaoH2010, Klose2005} employed the steady-state radiative transfer equation (RTE) as a forward model to reconstruct internal isotropic light sources from angularly resolved boundary measurements of photon flux. Theoretically, such angle-dependent measurements can compensate for the lack of spatial information and ensure uniqueness of the solution (\cite{Bal2016}). However, solving steady-state RTE is computationally demanding, and acquiring angularly resolved measurements is technically challenging in practice, both of which limit the feasibility of this approach for real-world biomedical applications. Gong et al. (\cite{Gong2024}) proposed a time-for-space strategy and established a coupled dynamic model with hybrid regularization, proving the theoretical uniqueness of the solution under this framework. Nevertheless, in practical applications, the limited diversity of boundary observations and the rapid temporal decay of solutions to parabolic governing equations severely restrict the numerical accuracy of source reconstruction. Consequently, although the location and coarse shape of high-dimensional sources can often be recovered, precise delineation of the source region and accurate recovery of the intensity distributions remain challenging.

The second strategy does not seek full quantitative reconstruction of the source; rather, it focuses on identifying key features, such as location (e.g., centroid) or geometric characteristics (e.g., shape or volume). This approach improves the identifiability of inverse problems by reducing the number of unknowns. Recently, various methods have been proposed to recover the spatial position and morphological information of sources using steady-state BLT models (\cite{Gao2017, Guo2020, Liu2022, Chen2022, Chu2023}).

The third strategy seeks to mitigate non-uniqueness by incorporating a priori information about the unknown source, thereby reducing the effective degrees of freedom in the reconstruction. Such a priori information often comprises source structural assumptions and permissible source regions (PSR). For instance, uniqueness can be established when the source is modeled as a single point or as a linear combination of point sources (\cite{WangG2004, Jiang2007}). Moreover, if the source is represented as $q = \phi \chi_{\Omega_0}$ with a known support $\Omega_0$, then the intensity function $\phi$ can be well reconstructed. The reconstruction accuracy is closely related to how well $\Omega_0$ approximates the true support. PSR can often be estimated using complementary imaging modalities, such as computed tomography (CT) or magnetic resonance imaging (MRI) (\cite{Klose2010, Zhang2014, Deng2020, Yin2022}). Recently, Ding et al. (\cite{Ding2024}) extended these results by proving the uniqueness of source regions with piecewise constant intensities in more general geometries, including domains with $C^2$ boundaries and polygonal domains, thereby generalizing earlier uniqueness results for point sources (\cite{WangG2004}).

Motivated by these unique results, subsequent studies developed shape optimization methods that focus on reconstructing the source region(\cite{Fan2023, Afraites2022, Charkaoui2021, GongWei2025}). Gong et al. (\cite{GongWei2025}) proposed an optimization model that treats the source region as the sole control variable. The key idea is to eliminate the source intensity by expressing the Tikhonov-regularized intensity in terms of the adjoint state for a given source region, thereby reducing the problem to a shape optimization model involving only the source region. Note that when the intensity is represented by the adjoint variable, the choice of regularization parameter becomes critical.

%and steady-state DA-based BLT 
Gong et al. introduced the coupled complex boundary method (CCBM) and its parameter-dependent variant for elliptic inverse source problems (\cite{Gong2016ccbm, Gong2016rte, Cheng2014, Cheng2015, Gong2020}). The CCBM offers several advantages. In particular, its parameter-dependent form obviates the need to select a regularization parameter (\cite{Gong2016ccbm, Gong2020}), thereby simplifying the optimization process. Rabago et al. (\cite{Rabago2025, Rabago2024}) combined the CCBM with shape optimization and achieved notable results in free boundary and obstacle detection problems. Afraites et al. (\cite{Afraites2022ccbmiop, Afraites2022ccbmisp}) applied the CCBM to geometric inverse source and obstacle problems, demonstrating its advantages over classical methods such as Kohn–Vogelius and least-squares approaches. More recently, Hrizi et al. (\cite{Hrizi2025}) extended the CCBM to time-fractional subdiffusion equations.

In this study, we integrate the CCBM with shape optimization techniques for the steady-state DA-based BLT problem. By leveraging the advantage of CCBM in obviating the selection of a regularization parameter, we seek to mitigate the sensitivity of intensity reconstruction to uncertainties in region identification. This approach is expected to enable unique reconstruction of the source region and, consequently, improve the overall imaging quality.

The main contributions of this study are summarized as follows:
\begin{enumerate}[(1)] \vspace{-0.2cm}
	\item By introducing an adjoint variable, the source support is decoupled from its intensity, thereby reducing the reconstruction problem for the bioluminescent source to a shape optimization problem involving only the source region. \vspace{-0.2cm}
	\item We employ a parameter-dependent CCBM algorithm that obviates the need to select a regularization parameter, resulting in more robust and stable numerical reconstructions. \vspace{-0.2cm}
	\item For piecewise constant nested sources, unlike \cite{Ding2024}, which relies on a priori knowledge of the source geometry and iteratively adjusts only a few prescribed parameters (e.g., centroid and radius), our method updates the geometry dynamically through a level-set representation, thereby achieving a genuine and accurate reconstruction of the source region. 
\end{enumerate}\vspace{-0.2cm}

The remainder of this paper is organized as follows. Section~\ref{sec:model} reformulates the parameter-dependent CCBM-based inverse source problem in BLT into a shape optimization problem involving only the source region and establishes the well-posedness of the resulting formulation. Section~\ref{sec:shape} presents the analysis of the shape differentiability of the state system with respect to domain perturbations and provides a detailed sensitivity analysis. Section~\ref{sec:algorithm} presents a two-step reconstruction algorithm that integrates a level set based geometric update with a parameter-dependent CCBM scheme for source intensity recovery. Section~\ref{sec:experiment} reports numerical experiments that demonstrate the effectiveness, robustness, and practical applicability of the proposed approach. Finally, concluding remarks are presented in Section~\ref{sec:concl}. Throughout this study, we adopt the standard notation for Sobolev spaces and their associated norms.

\section{Reconstruction via a shape optimization model} \label{sec:model}
We begin by introducing notations for the function spaces and stating assumptions regarding the problem data. For any set $G$ (e.g., $\Omega$, $\Gamma$, or $\Omega_0$), let $W^{m,s}(G)$ denote the standard real-valued Sobolev space with norm $\| \cdot \|_{m,s,G}$, and set $W^{0,s}(G) := L^s(G)$. In particular, $H^m(G)$ denotes $W^{m,2}(G)$, equipped with the inner product $(\cdot, \cdot)_{m,G}$ and the norm $\| \cdot \|_{m,G}$. The corresponding complex-valued space is denoted by $\mathbf{H}^m(G)$; its inner product is defined by $((u, v))_{m,G} := (u, \bar{v})_{m,G}$ and its norm by $\interleave u \interleave_{m,G}^2 := ((u, u))_{m,G}$. The source intensity $\phi$ is sought in an admissible set $Q_{\mathrm{ad}} \subset L^2(\Omega)$, which is assumed to be non-empty, closed, and convex. The domain $\Omega \subset \mathbb{R}^d$ ($d \leq 3$) is assumed to be open, bounded, and to have a Lipschitz boundary $\Gamma$. We assume that the coefficients and data satisfy the following conditions: $D \in L^\infty(\Omega)$ with $D \geq D_0 > 0$ a.e. in $\Omega$; $\mu_a \in L^\infty(\Omega)$ with $\mu_a \geq 0$ a.e. in $\Omega$; $A(x) \in [A_l, A_u]$ for constants $0 < A_l < A_u < \infty$; and $g_1, g_2 \in L^2(\Gamma)$. Throughout this paper, $C$ denotes a generic positive constant whose value may change from line to line.

\subsection{The uniqueness results and a shape optimization method}\label{subsec:blt_model}
The BLT problem aims to quantitatively reconstruct the spatial distribution of bioluminescent sources inside a small animal by using optical signals measured on the body surface. Owing to the strong scattering nature of biological tissues and limited boundary observations, the BLT inverse problem is inherently ill-posed. More precisely, the non-uniqueness of solutions may be characterized as follows:

\begin{proposition}[\cite{WangG2004}]
	Suppose Problem~\ref{prob:blt} admits a solution. Then there exists a representative solution $q_{H}$ of the minimal $L^2$ norm, and every solution $q$ can be written as:
	\begin{equation*}
		q = q_{H} - \Delta m + m \quad \forall m \in H^2_0(\Omega),
	\end{equation*}
	where $H^2_0(\Omega)$ is the closure of all smooth functions in $\Omega$ and vanishes on $\Gamma$ up to order one.
\end{proposition}

To address the non-uniqueness in the general setting, we restrict the class of admissible sources to a parametric form. This restriction enables the uniqueness to be established within a specific framework. By introducing a finite number of parameters to characterize the source, the inverse problem is reduced to a tractable subset of distributions for which uniqueness can be proved. Under this parameterization, the source is represented by
\begin{equation}
	q = \phi \chi_{\Omega_0},
	\label{phi_para}
\end{equation}
where $\Omega_0 \subset \Omega$ denotes the support of the source, $\phi$ is the intensity, and $\chi_{\Omega_0}$ is the characteristic function of $\Omega_0$. Note that $\Omega_0$ may consist of a union of several disjoint subdomains of $\Omega$. 

When we restrict attention to the source distributions of form \eqref{phi_para}, the uniqueness of the solution can be established.
\begin{proposition}[\cite{Ding2024}]\label{prop:unique}
	Let $\Omega \subset \mathbb{R}^d$ ($d = 2,3$) be a bounded open set and $\Omega \setminus \Omega_0$ is connected. Suppose $q \in L^2(\Omega)$ solves Problem~\ref{prob:blt} and admits the representation $q = \phi \chi_{\Omega_0}$, with $\Omega_0 \Subset \Omega$ a bounded open subset. Assuming either \vspace{-0.2cm}
	\begin{enumerate}[(a)]
		\item $\partial\Omega_0 \in C^2$, $\phi \in C^1(\bar{\Omega}_0)$, and $\phi \neq 0$ on $\partial\Omega_0$; or \vspace{-0.2cm}
		\item $\Omega_0$ is a convex polygon (or polyhedron) with corners $x_c$, $\phi$ is $C^\gamma(\bar{\Omega}_0)$-H\"{o}lder continuous for some $\gamma \in (0,1)$, and $\phi(x_c) \neq 0$ at each corner $x_c$. \vspace{-0.2cm}
	\end{enumerate}
	Then, $q$ is uniquely determined by the single boundary measurement $g$: in case (a), both the smooth domain $\Omega_0$ and the boundary values of $\phi$ on $\partial\Omega_0$ are uniquely determined; in case (b), both the polygonal (polyhedral) domain $\Omega_0$ and corner values $\phi(x_c)$ are uniquely determined. 
	In particular, if $\phi$ is constant, both $\Omega_0$ and $\phi$ are uniquely determined from the single boundary measurement~$g$.
\end{proposition}

Given the above unique results, we now solve the BLT problem. For simplicity, we denote $g_1 = -g$ and $g_2 = 2 A g$. Then Problem~\ref{prob:blt} is reduced to the following form:
\begin{problem}\label{prob:blt_para}
	Given $g_1, g_2$ on $\Gamma$, find $\Omega_0$ and $\phi$ such that
	\begin{equation}
		\left \{
		\begin{aligned}
			- {\rm div} (D \nabla u) + \mu_{a} u = \phi \chi_{\Omega_0} & \quad \mathrm{in}\ \Omega,\\[0.5em]
			u = g_2, \quad D \partial_n u= g_1 & \quad \mathrm{on}\ \Gamma.
		\end{aligned}
		\right.
		\label{eq:bvp_reduced}
	\end{equation}
\end{problem} 

In the following, we allow Neumann and Dirichlet data $g_1$ and $g_2$ to be contaminated by random noise of a known level $\delta$. Given a known subregion $\Omega_0 \subset \Omega$ and boundary data $g_1, g_2$, identification of the unknown source $\phi$ becomes relatively straightforward. A common approach is the Tikhonov regularization, which leads to the following formulation:
\begin{equation*}
	\min_{\phi \in Q_{ad}} \tilde{J_0}(\phi, u) = \dfrac{1}{2} \Vert u-g_2 \Vert_{0, \Gamma}^2 + \dfrac{\varepsilon}{2} \Vert \phi \Vert_{0, \Omega_0}^2,
	\label{prob:phi}
\end{equation*} 
subject to
\begin{equation}
	\left \{
	\begin{aligned}
		- \textrm{div} (D \nabla u) + \mu_{a} u &= \phi \chi_{\Omega_0} & \mathrm{in}\ \Omega,\\[0.5em]
		D \partial_n u &= g_1  & \mathrm{on}\ \Gamma,
	\end{aligned}
	\right.
	\label{eq:bvp_phi}
\end{equation}
where $\varepsilon > 0$ is the regularization parameter. This formulation and its variants have been studied extensively (\cite{Cheng2014, Han2006}). Subsequently, the adjoint variable $w \in H^1(\Omega)$ is introduced as follows:
\begin{equation*}
	\left \{
	\begin{aligned}
		- \textrm{div} (D \nabla w) + \mu_{a} w &= 0 & \quad \mathrm{in}\ \Omega,\\[0.5em]
		D \partial_n w &= u-g_2 & \quad \mathrm{on}\ \Gamma.
	\end{aligned}
	\right.
\end{equation*}
Then the first-order optimality condition reads:
\begin{equation}
	\nabla \tilde{J_0} (\phi) = \varepsilon \phi + w = 0 \quad \textrm{in} \ \Omega_0,
	\label{eq:1st_cond}
\end{equation}
which yields the explicit expression $\phi = -\tfrac{1}{\varepsilon} w |_{\Omega_0}$. The well-posedness, finite element discretization, and theoretical analysis of this formulation are detailed in \cite{Han2006}. As indicated by condition \eqref{eq:1st_cond}, once the source region $\Omega_0$ is known, the source intensity $\phi$ can be reconstructed. In practice, however, the region $\Omega_0$ is usually unknown and represents the primary quantity of interest in the reconstruction.

In this study, we aim to recover both the support $\Omega_0$ and the intensity $\phi$ of the source. Building on the foregoing formulation, we consider the following joint optimization problem:
\begin{equation*}
	\min_{\Omega_0 \subset \Omega, \ \phi \in Q_{ad}} \tilde{J_0}(\Omega_0, \phi, u) = \dfrac{1}{2} \| u-g_2 \|_{0, \Gamma}^2 + \dfrac{\varepsilon}{2} \| \phi \|_{0, \Omega_0}^2 \quad
	\textrm{subject to}\ \ \eqref{eq:bvp_phi}.
\end{equation*} 
To eliminate the dependence on the unknown source $\phi$, we substitute the optimality condition $\phi = - \tfrac{1}{\varepsilon} w |_{\Omega_0}$ into the objective functional (\cite{GongWei2025}). This yields the following shape optimization problem:
\begin{equation}
	\min_{\Omega_0 \subset \Omega} \tilde{J_1}(\Omega_0, u, w) =  \frac{1}{2} \Vert u-g_2 \Vert_{0,\Gamma}^2 + \frac{1}{2\varepsilon} \Vert w \Vert_{0,\Omega_0}^2, 
	\label{eq:shape_p}
\end{equation}
subject to
\begin{equation}
	\left \{
	\begin{aligned}
		- \textrm{div} (D \nabla u) + \mu_{a} u &= -\frac{1}{\varepsilon} w \chi_{\Omega_0} & \quad \textrm{in}\ \Omega,\\[0.5em]
		D \partial_n u &= g_1 & \quad \textrm{on}\ \Gamma,\\[0.5em]
		- \textrm{div} (D \nabla w) + \mu_{a} w &= 0 & \quad \textrm{in}\ \Omega,\\[0.5em]
		D \partial_n w &= u-g_2 & \quad \textrm{on}\ \Gamma.
	\end{aligned}
	\right.
	\label{eq:state_system}
\end{equation}
Thus, the control variable $\phi$ is eliminated, and the original optimization problem over $(\phi, \Omega_0)$ with the state variable $u$ is reduced to an equivalent shape optimization problem over $\Omega_0$ with the state and adjoint variables $(u, w)$.

The corresponding theoretical analysis and numerical experiments are presented in detail in \cite{GongWei2025}.

\subsection{A new CCBM-based regularized shape optimization method} \label{subsec:CCBM}

When the source region $\Omega_0$ is prescribed, the source intensity $\phi$ can be determined from the first-order optimality condition \eqref{eq:1st_cond}. It is worth noting that the choice of regularization parameter $\varepsilon$ has a substantial impact on the accuracy of the recovered intensity. As shown in \cite{Gong2016ccbm}, the CCBM approach circumvents the need to select this parameter, which also offers improved robustness and requires weaker regularity assumptions for the data.

In this study, we incorporate the decoupling strategy for the source region and intensity proposed in \cite{GongWei2025} into the CCBM framework, thereby developing a new regularized shape optimization method. Specifically, within the CCBM framework, the source intensity $\phi$ is reconstructed by solving the following optimization problem:
\begin{equation}
	\min_{\phi \in Q_{ad}} J_0(\phi, u) = \dfrac{1}{2} \Vert u_2 \Vert_{0, \Omega}^2 + \dfrac{\varepsilon}{2} \Vert \phi \Vert_{0, \Omega_0}^2
	\label{prob:J_phi}
\end{equation} 
subject to
\begin{equation}
	\left \{
	\begin{aligned}
		- \textrm{div} (D \nabla u) + \mu_{a} u  &= \phi \chi_{\Omega_0} & \quad \textrm{in}\ \Omega,\\[0.5em]
		D \partial_n u + i \alpha u &= g_1 + i \alpha g_2 & \quad \textrm{on}\ \Gamma,
	\end{aligned}
	\right.
	\label{eq:ccbm}
\end{equation}
where $\alpha > 0$ is a parameter and $i$ is the imaginary unit. Because Eqn.\eqref{eq:ccbm} is complex, the standard Lagrange multiplier method cannot be directly applied to derive the optimality conditions. Therefore, we reformulated the problem within a real-valued shape optimization framework. Substituting $u = u_1 + i u_2$ into \eqref{eq:ccbm} and separating the real and imaginary components, we derive a real-valued coupled system as follows:
\begin{equation}
	\left \{
	\begin{aligned}
		- \textrm{div} (D \nabla u_1) + \mu_{a} u_1  &= \phi \chi_{\Omega_0} & \quad \textrm{in}\ \Omega,\\[0.5em]
		D \partial_n u_1 - \alpha u_2 &= g_1 & \quad \textrm{on}\ \Gamma,
	\end{aligned}
	\right.
	\label{eq:ccbm_1}
\end{equation}
and
\begin{equation}
	\left \{
	\begin{aligned}
		- \textrm{div} (D \nabla u_2) + \mu_{a} u_2  &= 0 & \quad \textrm{in}\ \Omega,\\[0.5em]
		D \partial_n u_2 + \alpha u_1 &= \alpha g_2 & \quad \textrm{on}\ \Gamma.
	\end{aligned}
	\right.
	\label{eq:ccbm_2}
\end{equation}
The corresponding weak formulations are given by
\begin{equation}
	0 = (D \nabla u_1, \nabla v)_{\Omega} +  (\mu_a u_1, v)_{\Omega} - (\phi, v)_{\Omega_0} - (g_1 + \alpha u_2, v)_{\Gamma} \triangleq f_1(u_1,u_2,v) \quad \forall v \in H^1(\Omega),
	\label{eq:ccbm_weak_1} 
\end{equation}
\begin{equation}
	0 = (D \nabla u_2, \nabla v)_{\Omega} + (\mu_a u_2, v)_{\Omega} - \alpha (g_2 - u_1, v)_{\Gamma} \triangleq f_2(u_1,u_2,v) \quad \forall v \in H^1(\Omega).
	\label{eq:ccbm_weak_2}
\end{equation}
The associated optimization problem can then be formulated as
\begin{equation}
	\min_{\Omega_0 \in \Omega, \phi \in Q_{ad}} J_0(\Omega_0, \phi, u_1, u_2) = \dfrac{1}{2} \Vert u_2 \Vert_{0, \Omega}^2 + \dfrac{\varepsilon}{2} \Vert \phi \Vert_{0, \Omega_0}^2 
	\quad \textrm{subject to \eqref{eq:ccbm_1} and  \eqref{eq:ccbm_2}}.
	\label{prob:J_omega&phi}
\end{equation}

Next, we construct the Lagrangian functional $\mathcal{L}_1$ by incorporating \eqref{eq:ccbm_weak_1}-\eqref{eq:ccbm_weak_2} into \eqref{prob:J_omega&phi}:
\begin{equation*}
	\mathcal{L}_1(\phi, u_1, u_2, w_1, w_2) = \dfrac{1}{2} \Vert u_2 \Vert_{0, \Omega}^2 + \dfrac{\varepsilon}{2} \Vert \phi \Vert_{0, \Omega_0}^2 + f_1(u_1,u_2,w_1) + f_2(u_1,u_2,w_1),
\end{equation*}
where $w_1$ and $w_2$ are the adjoint variables corresponding to $u_1$ and $u_2$, and serve as Lagrange multipliers. By taking the partial derivatives of $\mathcal{L}_1$ with respect to all variables and setting them to zero, we obtain adjoint equations and first-order optimality conditions. For any $\nu_1 \in L^2(\Omega)$ and $\nu_2 \in H^1(\Omega)$,
\begin{equation*}
	\frac{\partial \mathcal{L}_1}{\partial \phi} \cdot \nu_1 = 0, \quad
	\frac{\partial \mathcal{L}_1}{\partial u_1} \cdot \nu_2 = 0, \quad
	\frac{\partial \mathcal{L}_1}{\partial u_2} \cdot \nu_2 = 0, \quad
	\frac{\partial \mathcal{L}_1}{\partial w_1} \cdot \nu_2 = 0, \quad
	\frac{\partial \mathcal{L}_1}{\partial w_2} \cdot \nu_2 = 0.
\end{equation*}
Consequently, we obtain
\begin{equation}
	\varepsilon (\phi, \nu_1)_{\Omega_0} - (w_1, \nu_1)_{\Omega_0} = 0, \quad \forall \nu_1 \in L^2(\Omega),
	\label{eq:lag_0}
\end{equation}
\begin{equation}
	(D \nabla w_1, \nabla\nu_2)_{\Omega} + (\mu_a w_1, \nu_2)_{\Omega} + \alpha (w_2, \nu_2)_{\Gamma} = 0 \quad \forall \nu_2 \in H^1(\Omega),
	\label{eq:lag_3}
\end{equation}
\begin{equation}
	(D \nabla w_2, \nabla\nu_2)_{\Omega} + (\mu_a w_2, \nu_2)_{\Omega} - \alpha (w_1, \nu_2)_{\Gamma} + \alpha (u_2, \nu_2)_{\Omega}= 0 \quad \forall \nu_2 \in H^1(\Omega),
	\label{eq:lag_4}
\end{equation}
\begin{equation}
	(D \nabla u_1, \nabla\nu_2)_{\Omega} + (\mu_a u_1, \nu_2)_{\Omega} -  (g_1 + \alpha u_2, \nu_2)_{\Gamma} - (\phi, \nu_2)_{\Omega_0} = 0 \quad \forall \nu_2 \in H^1(\Omega),
	\label{eq:lag_1}
\end{equation}
\begin{equation}
	(D \nabla u_2, \nabla\nu_2)_{\Omega} + (\mu_a u_2, \nu_2)_{\Omega} - \alpha (g_2 - u_1, \nu_2)_{\Gamma} = 0 \quad \forall \nu_2 \in H^1(\Omega).
	\label{eq:lag_2}
\end{equation}
Eqns.\eqref{eq:lag_3} and \eqref{eq:lag_4} define the adjoint variables $w_1$ and $w_2$, whereas Eqns.\eqref{eq:lag_1} and \eqref{eq:lag_2} correspond to the weak forms of \eqref{eq:ccbm_1} and \eqref{eq:ccbm_2}, respectively. From Eqn.\eqref{eq:lag_0}, the source intensity can be written as $\phi = \tfrac{1}{\varepsilon} w_1 |_{\Omega_0}$. Substituting this relation into \eqref{prob:J_omega&phi} yields the following shape optimization problem:
\begin{equation}
	\min_{\Omega_0 \subset \Omega} J_1(\Omega_0) =  \dfrac{1}{2} \Vert u_2 \Vert_{0,\Omega}^2 + \dfrac{1}{2\varepsilon} \Vert w_1 \Vert_{0,\Omega_0}^2, 
	\label{eq:shape optimization problem}
\end{equation}
subject to \eqref{eq:lag_3}-\eqref{eq:lag_2} or
\begin{equation}
	\left \{
	\begin{aligned}
		- \textrm{div} (D \nabla u) + \mu_{a} u  &= \frac{1}{\varepsilon} w_1 \chi_{\Omega_0} & \quad \textrm{in}\ \Omega,\\[0.5em]
		D \partial_n u + i \alpha u &= g_1 + i \alpha g_2 & \quad \textrm{on}\ \Gamma,\\[0.5em]
		- \textrm{div} (D \nabla w) + \mu_{a} w &= -i u_2 & \quad \textrm{in}\ \Omega,\\[0.5em]
		D \partial_n w + i \alpha w &= 0 & \quad \textrm{on}\ \Gamma.
	\end{aligned}
	\right.
	\label{eq:couple system}
\end{equation}
The weak formulation of the coupled system \eqref{eq:couple system} is as follows: find  $(u,w) \in \mathbf{H}^1(\Omega) \times \mathbf{H}^1(\Omega)$ such that
\begin{equation}
	\left \{
	\begin{aligned}
		&\int_{\Omega}(D \nabla u \cdot \nabla v + \mu_{a} uv)dx - \dfrac{1}{\varepsilon} \int_{\Omega_0} w_1 v dx + \int_{\Gamma} i \alpha uv ds = \int_{\Gamma} (g_1 + i \alpha g_2)v ds, & \quad \forall v \in \mathbf{H}^1(\Omega),\\[1em]
		&\int_{\Omega}(D \nabla w \cdot \nabla s + \mu_{a} ws)dx + \int_{\Omega} i u_2 s dx + \int_{\Gamma} i \alpha ws ds = 0, & \quad \forall s \in \mathbf{H}^1(\Omega).
	\end{aligned}
	\right.
	\label{eq:couple weak}
\end{equation}
This weak system \eqref{eq:couple weak} represents the first-order optimality conditions of a strictly convex optimization problem for recovering $\phi$ with a prescribed region $\Omega_0$ and given boundary data. It admits a unique solution pair $(u, w) \in \mathbf{H}^1(\Omega) \times \mathbf{H}^1(\Omega)$.

In the next section, we investigate the well-posedness of the resulting shape optimization problem and derive its shape sensitivity analysis.

\subsection{Well-posedness analysis for the new optimization problem} \label{subsec:well-posed}
To establish the stability of the coupled system \eqref{eq:couple system}—a key component for proving the differentiability of the solution with respect to the domain $\Omega_0$—we first introduce appropriate rescaling and reformulate the system in a variational form. This leads to a symmetric saddle-point structure, which facilitates the subsequent well-posedness analysis. Define
\begin{equation*}
	\mathbf{X} := \mathbf{H}^1(\Omega) \times \mathbf{H}^1(\Omega) \quad \textrm{with} \quad \Vert \boldsymbol{v} \Vert := (\interleave v_1 \interleave_{1,\Omega}^2 + \interleave v_2 \interleave_{1,\Omega}^2)^{\frac{1}{2}} \quad \forall \boldsymbol{v}=(v_1,v_2) \in \mathbf{X}.
\end{equation*}
Introduce the bilinear form $\boldsymbol{b}: \mathbf{X} \times \mathbf{X} \to \mathbb{R}$ defined as
\begin{equation*}
	\boldsymbol{b}(\boldsymbol{v},\boldsymbol{w}) := \boldsymbol{a}(\boldsymbol{v},\boldsymbol{w}) + \frac{1}{\sqrt{\varepsilon}} \boldsymbol{c}(\boldsymbol{v},\boldsymbol{w}) + i \alpha \boldsymbol{d}(\boldsymbol{v},\boldsymbol{w}),
\end{equation*}
where
\begin{equation*}
	\begin{aligned}
		&\boldsymbol{a}(\boldsymbol{v},\boldsymbol{w}) := a(v_1, w_1) + a(v_2, w_2), &\forall \boldsymbol{v},\boldsymbol{w} \in \mathbf{X},\\[0.5em]
		&a(u,v) := (D \nabla u, \nabla v) +  (\mu_{a} u,v), &\forall u,v \in \mathbf{H}^1(\Omega),\\[0.5em]
		&\boldsymbol{c}(\boldsymbol{v},\boldsymbol{w}) := (\textrm{Im } v_2, \textrm{Im } w_1)_{\Omega_0} - (v_1, w_2)_{\Omega}, &\forall \boldsymbol{v},\boldsymbol{w} \in \mathbf{X},\\[0.5em]
		&\boldsymbol{d}(\boldsymbol{v},\boldsymbol{w}) := (v_1, w_1)_{\Gamma} + (v_2,w_2)_{\Gamma}, &\forall \boldsymbol{v},\boldsymbol{w} \in \mathbf{X}.
	\end{aligned}
\end{equation*}
Subsequently, the weak formulation \eqref{eq:couple weak} can be rewritten in the following compact form:
\begin{equation}
	\textrm{find } \boldsymbol{x} \in \mathbf{X} \quad \textrm{s.t.} \  \boldsymbol{b}(\boldsymbol{x},\boldsymbol{\phi}) = (g_1 + i \alpha g_2, \phi_1)_{\Gamma}, \qquad \forall \boldsymbol{\phi} \in \mathbf{X},
	\label{weak}
\end{equation}
where $\boldsymbol{x}=(x_1,x_2) \in \mathbf{X}$ with $x_1 = u$ and $x_2 = \frac{1}{\sqrt{\varepsilon}}w$ solving \eqref{eq:couple weak}.

For a given $\phi \in L^2(\Omega_0)$, using the complex version of the Lax-Milgram Lemma, problem (\ref{weak}) has a unique solution $\boldsymbol{x} \in \mathbf{X}$. Moreover, through a similar deduction (\cite{GongWei2025}), we obtain the stability result
\begin{equation}
	\interleave u \interleave_{1,\Omega}+\frac{1}{\sqrt{\varepsilon}} \interleave w \interleave_{1,\Omega} \leq C (\Vert g_1 \Vert_{0,\Gamma} + \alpha \Vert g_2 \Vert_{0,\Gamma}),
	\label{eq:stability_result}
\end{equation}
where the constant $C>0$ is independent of $\alpha$.

In fact, shape optimization problems are known to be inherently ill-posed; in particular, solutions may not exist or may lack uniqueness (\cite{Allaire2021}). A standard approach to address this issue is to augment the objective functional with a regularization term. Accordingly, we introduce the following regularized optimization problem:
\begin{equation}
	\min_{\Omega_0 \subset \Omega} J(\Omega_0) = J_1(\Omega_0)+\lambda \mathcal{P}_{\Omega}(\Omega_0) \quad \textrm{s.t.} \ \eqref{eq:couple system},\label{eq:shape_opti}
\end{equation}
where $\lambda > 0$ is a regularization parameter, and $\mathcal{P}_{\Omega}(\Omega_0)$ denotes the perimeter of $\Omega_0$ in $\Omega$, which is defined as
\begin{equation*}
	\mathcal{P}_{\Omega}(\Omega_0) := \sup_{\phi \in \Phi} \int_{\Omega_0} \textrm{div} \phi, \quad \text{with} \quad
	\Phi := \{\phi \in C_c^1(\Omega, \mathbb{R}^d) : \max_{x \in \Omega} \Vert \phi(x) \Vert \leq 1 \}.
\end{equation*}

The inclusion of the perimeter term $\mathcal{P}_{\Omega}(\Omega_0)$ promotes the regularity of the domain and ensures the existence of an optimal solution to the shape optimization problem.

\begin{theorem}
	The regularized shape identification problem \eqref{eq:shape_opti} admits at least one solution.
\end{theorem}
\begin{proof}[Sketch of the proof]
	We restrict our attention to the case of domains with a finite perimeter $\mathcal{P}_{\Omega}(\Omega_0)$. Define $j(\chi_{\Omega_0}):=J(\Omega_0)$. Take a minimizing sequence $\{\chi_n\} \subset \textrm{Char}(\Omega,M)$ for some constant $M>0$ such that
	\begin{equation*}
		\lim_{n\to\infty} j(\chi_n) = \inf_{\chi \in \textrm{Char}(\Omega,M)} j(\chi),
	\end{equation*} 
	where
	\begin{equation*}
		\textrm{Char}(\Omega,M) = 
		\Big\{ \chi_{\Omega_0} \in L^2(\Omega) \; : \;
		\chi_{\Omega_0}(1-\chi_{\Omega_0}) = 0 \ \text{a.e. in } \Omega, \;
		\mathcal{P}_\Omega(\Omega_0) \leq M
		\Big\}.
	\end{equation*}	
	Owing to the compactness of $\textrm{Char}(\Omega,M)$, there exists a subsequence $\{\chi_{n_k}\}$ that converges weakly to some $\chi_0 \in \textrm{Char}(\Omega,M)$. By the lower semicontinuity of $J$ (\cite{Sokolowski1992}, Lemma 2.6), we obtain
	\begin{equation*}
		j(\chi_0) \leq \liminf_{k \to \infty} j(\chi_{n_k}) = \inf_{\chi \in \textrm{Char}(\Omega,M)} j(\chi),
	\end{equation*} 
	which shows that $\chi_0$ attains a minimum value. $\hfill\square$	
\end{proof}

We now demonstrate that the regularized solution remains uniformly stable with respect to the regularization parameter, similar to Theorem 3.2 of \cite{Gong2016ccbm}.
\begin{proposition} \label{prop:select_alpha}
	Let $\alpha = \mathcal{O}(\sqrt{\varepsilon})$. Then, for any fixed $ \delta \geq 0$, $\phi = \tfrac{1}{\varepsilon} w_1 \chi_{\Omega_0}$ remains uniformly bounded in $ L^2 (\Omega_0)$ with respect to $ \varepsilon $ for a sufficiently small $ \varepsilon > 0$, where $w_1$ is the real part of the weak solution $w \in \mathbf{H}^1(\Omega)$ of the adjoint problem
	\begin{equation}
		\left \{
		\begin{aligned}
			- {\rm div} (D \nabla w) + \mu_{a} w &= -i u_2 & \quad {\rm in}\ \Omega,\\[0.5em]
			D \partial_n w + i \alpha w &= 0 & \quad {\rm on}\ \Gamma.
		\end{aligned}
		\right.
		\label{eq:adjoint_problem}
	\end{equation}
\end{proposition}

\begin{proof}
	By \eqref{eq:stability_result} and $\alpha = \mathcal{O}(\sqrt{\varepsilon})$, we have
	\begin{equation*}
		\interleave u \interleave_{1,\Omega} \leq C (\Vert \phi \Vert_{0,\Omega_0} + \Vert g_1 \Vert_{0,\Gamma} + \alpha \Vert g_2 \Vert_{0,\Gamma}) \leq C.
	\end{equation*}
	The weak form of \eqref{eq:ccbm_2} is
	\begin{equation}
		D (\nabla u_2, \nabla v)_{1, \Omega} + \mu_a (u_2, v)_{1, \Omega} = \alpha (g_2-u_1, v)_{0,\Gamma} \quad \forall v \in H^1(\Omega).
		\label{eq:weak_u2}
	\end{equation}
	Taking $v=u_2$ in \eqref{eq:weak_u2}, we get
	\begin{equation}
		\Vert u_2 \Vert_{1, \Omega} \leq C \alpha (\Vert g_2 \Vert_{0, \Gamma} + \Vert u_1 \Vert_{0, \Gamma}) \leq C \alpha.
	\end{equation}
	From the adjoint problem \eqref{eq:adjoint_problem}, we have
	\begin{equation*}
		\interleave w \interleave_{1,\Omega} \leq C \Vert u_2 \Vert_{0,\Omega} \leq C \alpha.
	\end{equation*}
	Then $w_1$ solves
	\begin{equation}
		D (\nabla w_1, \nabla v)_{1, \Omega} + \mu_a (w_1, v)_{1, \Omega} = \alpha (w_2, v)_{0,\Gamma} \quad \forall v \in H^1(\Omega).
		\label{eq:weak_p1}
	\end{equation}
	Taking $v=w_1$ in \eqref{eq:weak_p1}, we get
	\begin{equation}
		\Vert w_1 \Vert_{1, \Omega} \leq C \alpha \Vert w_2 \Vert_{0, \Gamma} \leq C \alpha^2.
	\end{equation}
	Therefore, if $\alpha = \mathcal{O}(\sqrt{\varepsilon})$,
	\begin{equation}
		\Vert \frac{1}{\varepsilon} w_1 \chi_{\Omega_0} \Vert_{0,\Omega_0} = \mathcal{O}(1).
	\end{equation}
	This completes the proof.$\hfill\square$	
\end{proof}

As indicated by Proposition \ref{prop:select_alpha}, the source function can be reliably reconstructed for sufficiently small regularization parameters when $\alpha$ is appropriately chosen. This observation provides practical guidance for selecting $\alpha$ in numerical implementations.

\section{Shape sensitivity analysis}\label{sec:shape}
Based on the theoretical analysis in the previous chapter, this section derives the shape derivative and associated gradient of the objective functional, which forms the basis for the shape optimization algorithm developed later.

\subsection{Shape derivatives}
Shape derivatives quantify the sensitivity of an objective functional to infinitesimal geometric perturbations in the domain. In this work, we adopt the perturbation of identity method, in which the deformation is represented as a first-order perturbation of the identity map generated by a smooth velocity field. This approach yields an explicit mapping between the reference and perturbed domains, thereby permitting a rigorous derivation of both material and shape derivatives. Compared with the velocity method, it avoids mesh regeneration during iterations, reduces computational cost, and improves numerical stability. Hence, it is widely used in PDE-constrained shape optimization problems. For further details on shape calculus, refer to \cite{Allaire2021, Delfour2011, Sokolowski1992}.

In this section, we compute the shape derivative of the shape optimization problem \eqref{eq:shape optimization problem}-\eqref{eq:couple system} using standard techniques from shape sensitivity analysis. First, we introduce some definitions. Let $\Omega$ be an open bounded domain with Lipschitz boundary. Utilizing the perturbation of identity approach, we define the perturbed domain for each $t>0$ as
\begin{equation*}
	\Omega_t:=T_t(\Omega)[\boldsymbol{V}]=(\textrm{Id}+t\boldsymbol{V})(\Omega).
\end{equation*}
where $\boldsymbol{V}$ is a velocity field belonging to
\begin{equation*}
	\textbf{W}_0^{1,\infty}(\Omega) := \left\{ \boldsymbol{V} \in L^\infty (\mathbb{R}^d), \nabla \boldsymbol{V} \in L^\infty (\mathbb{R}^d)^{d \times d}, \boldsymbol{V}|_{\Gamma} = 0 \right\}.
\end{equation*}
The key point in the calculation of the shape derivative $J'(\Omega)$ is the definition of an appropriate notion of derivative for the mapping $\Omega \mapsto u$. For any function $u_t \in \mathbf{H}^1(\Omega_t)$ defined on the perturbed domain $\Omega_t$, and denoted by $u=u_0 \in \mathbf{H}^1(\Omega)$ the corresponding function on the reference domain, we define the material derivative and the shape derivative of $u$ at $x \in \Omega$ as follows:
\begin{equation*}
	\dot{u}(x) = \lim_{t \to 0^{+}} \frac{u_t(T_t(x))-u(x)}{t}, \quad u'(x) = \lim_{t \to 0^{+}} \frac{u_t(x)-u(x)}{t}.
\end{equation*}
If both derivatives exist, the chain rule yields the following relation for any $x \in \Omega$,
\begin{equation}
	u'(x) = \dot{u}(x) - \nabla u(x) \cdot \boldsymbol{V}(x).
	\label{eq:chain_rule}
\end{equation}
\begin{remark}
	The shape and material derivatives represent two complementary perspectives for describing variations induced by domain perturbations. The shape derivative measures the change of a field at a fixed spatial point, and is therefore convenient for sensitivity analysis in shape optimization. In contrast, the material derivative follows the material particles as the domain deforms and is natural for transport problems. When both derivatives exist, they are related by \eqref{eq:chain_rule}, which links spatial and material variations.
\end{remark}

\begin{proposition}\label{thm:existence of weak material derivative}
	Assume $g_1, g_2 \in L^2(\Gamma)$ and that $\Omega_0$ is of class $C^k(k \geq 2)$. Then for any velocity field $\boldsymbol{V} \in C_c^k(\Omega; \mathbb{R}^d)$, the weak material derivative of elliptic system \eqref{eq:couple system} in direction $\boldsymbol{V}$ exists.
\end{proposition}

The proof is omitted and a detailed derivation is provided in \cite{GongWei2025}. Under the conditions specified in Proposition \ref{thm:existence of weak material derivative}, the weak material derivatives $\dot{u}(\Omega; \boldsymbol{V})$ and $\dot{w}(\Omega; \boldsymbol{V})$ are well-defined in the Sobolev space $\mathbf{H}^1(\Omega)$. Moreover, if the products $\nabla u \cdot \boldsymbol{V}$ and $\nabla w \cdot \boldsymbol{V}$ also belong to $\mathbf{H}^1(\Omega)$—which is ensured when $u, w \in \mathbf{H}^2(\Omega)$ or the domain satisfies the assumption of Proposition \ref{prop:unique}—then the shape derivatives $u'$ and $w'$ along the direction of $\boldsymbol{V}$ exist and lie in $\mathbf{H}^1(\Omega)$.

We now derive the forms of $u'(\Omega; \boldsymbol{V})$ and $w'(\Omega; \boldsymbol{V})$. For a function $\psi (t,x) \in C([0,\tau];W_{\textrm{loc}}^{1,1}(\mathbb{R}^d)) \cap C^{1}([0,\tau];L_{\textrm{loc}}^{1}(\mathbb{R}^d))$, we define the function $F_{\boldsymbol{V}}(t):=\int_{\Omega_t} \psi (t,x) dx$. The shape derivative of $F_{\boldsymbol{V}}(t)$ at $t=0$ can then be calculated as
\begin{equation}
	dF(\Omega;\boldsymbol{V}) = dF_{\boldsymbol{V}}(0)=\int_{\Omega} (\psi'(0,x)+\textrm{div}(\psi(0,x) \boldsymbol{V})) dx.
	\label{eq:dF_V(0)}
\end{equation}
Consider the weak formulation of the state system defined in the perturbed domain $\Omega_t$. For any $v_t,s_t \in \mathbf{H}^1(\Omega_t)$,
\begin{equation*}
	\left \{
	\begin{aligned}
		&\int_{\Omega_t}(D\nabla u_t \cdot \nabla v_t + \mu_{a} u_t v_t)dx - \dfrac{1}{\varepsilon} \int_{{\Omega_0}_t} w_{1t} v_t dx + \int_{\Gamma_t} i \alpha u_t v_t ds = \int_{\Gamma_t} (g_1 + i \alpha g_2)v_t ds,\\[1em]
		&\int_{\Omega_t}(D\nabla w_t \cdot \nabla s_t + \mu_{a} w_t s_t)dx + \int_{\Omega_t} i u_{2t} s_t dx + \int_{\Gamma_t} i \alpha w_t s_t ds = 0.
	\end{aligned}
	\right.
	\label{eq:perturbed}
\end{equation*}
Differentiating the above identities with respect to $t$ at $t=0$ and using \eqref{eq:dF_V(0)} yields the following variational identities for $u'$ and $w'$:
\begin{equation*}
	\left \{
	\begin{aligned}
		&\int_{\Omega}(D\nabla u' \cdot \nabla v + \mu_{a} u'v)dx + \int_{\Omega}\textrm{div}((D \nabla u \cdot \nabla v + \mu_{a} uv)\boldsymbol{V})dx - \dfrac{1}{\varepsilon} \int_{\Omega_0} w_1' v dx\\[1em]
		&- \dfrac{1}{\varepsilon} \int_{\Omega_0} \textrm{div}(w_1 v \boldsymbol{V})dx + \int_{\Gamma} i \alpha u'v ds= 0 \qquad \forall v \in \mathbf{H}^1(\Omega), \\[1em]
		&\int_{\Omega}(D \nabla w' \cdot \nabla s + \mu_{a} w's)dx + \int_{\Omega}\textrm{div}((D \nabla w \cdot \nabla s + \mu_{a} ws)\boldsymbol{V})dx + \int_{\Omega} i u_2' s dx \\[1em]
		&+ \int_{\Omega} \textrm{div}(i u_2 s \boldsymbol{V}) dx + \int_{\Gamma} i \alpha w's ds = 0 \qquad \forall s \in \mathbf{H}^1(\Omega).
	\end{aligned}
	\right.
\end{equation*}
Using the divergence theorem, the identities above can be rewritten as
\begin{equation*}
	\left \{
	\begin{aligned}
		&\int_{\Omega}(D \nabla u' \cdot \nabla v + \mu_{a} u'v)ds + \int_{\Gamma}((D \nabla u \cdot \nabla v + \mu_{a} uv)(\boldsymbol{V} \cdot \boldsymbol{n}))dx - \dfrac{1}{\varepsilon} \int_{\Omega_0} w_1' v dx \\[1em]
		&- \dfrac{1}{\varepsilon} \int_{\partial \Omega_0} w_1 v (\boldsymbol{V} \cdot \boldsymbol{n})ds + \int_{\Gamma} i \alpha u'v ds= 0 \qquad \forall v \in \mathbf{H}^1(\Omega),\\[1em]
		&\int_{\Omega}(D \nabla w' \cdot \nabla s + \mu_{a} w's)ds + \int_{\Gamma}(D \nabla w \cdot \nabla s + \mu_{a} ws)(\boldsymbol{V} \cdot \boldsymbol{n})dx + \int_{\Omega} i u_2' s dx \\[1em]
		&+ \int_{\Gamma} i u_2 s (\boldsymbol{V} \cdot \boldsymbol{n}) ds + \int_{\Gamma} i \alpha w's ds = 0 \qquad \forall s \in \mathbf{H}^1(\Omega).
	\end{aligned}
	\right.
\end{equation*}
Since $\boldsymbol{V}|_{\Gamma} = 0$, the terms on $\Gamma$ involving $\boldsymbol{V}$ vanish, and we arrive at
\begin{equation*}
	\left \{
	\begin{aligned}
		&\int_{\Omega}(D \nabla u' \cdot \nabla v + \mu_{a} u'v)dx - \dfrac{1}{\varepsilon} \int_{\Omega_0} w_1' v dx - \dfrac{1}{\varepsilon} \int_{\partial \Omega_0} w_1 v (\boldsymbol{V} \cdot \boldsymbol{n})ds + \int_{\Gamma} i \alpha u'v ds= 0 \  &\forall v \in \mathbf{H}^1(\Omega),\\[1em]
		&\int_{\Omega}(D \nabla w' \cdot \nabla s + \mu_{a} w's)dx + \int_{\Omega} i u_2' s dx + \int_{\Gamma} i \alpha w's ds = 0 \  &\forall s \in \mathbf{H}^1(\Omega).
	\end{aligned}
	\right.
\end{equation*}
Consequently, the shape derivatives $(u',w')$ are characterized as unique solutions to the following coupled boundary value problem:
\begin{equation*}
	\left \{
	\begin{aligned}		
		- \textrm{div} (D \nabla u') + \mu_{a} u' - \frac{1}{\varepsilon}w_1'\chi_{\Omega_0} &= 0, &\textrm{in } \Omega,\\[0.5em]
		- \textrm{div} (D \nabla w') + \mu_{a} w' + i u_2' &= 0, &\textrm{in } \Omega,\\[0.5em]
		[\![\partial_n u']\!] &= \frac{1}{\varepsilon}w_1(\boldsymbol{V} \cdot \boldsymbol{n}), & \textrm{on } \partial \Omega_0,\\[0.5em]		
		D \partial_n u' + i \alpha u' &= 0, &\textrm{on } \Gamma,\\[0.5em]
		D \partial_n w' + i \alpha w' &= 0, &\textrm{on } \Gamma.		
	\end{aligned}
	\right.
\end{equation*}

\begin{remark}
	Although the final expression of the shape derivative $J'(\Omega)$ does not explicitly contain the shape derivatives $u'$ and $w'$, their existence ensures the differentiability of $J(\Omega)$ and justifies the use of adjoint variables to remove $u'$ and $w'$ from the final expression of the shape gradient.
\end{remark}

\subsection{Shape gradient of the objective functional}
We begin by deriving the shape gradient of the objective functional, which plays a central role in driving domain deformations within the optimization procedure. 

In many practical settings, constraints, such as enforcing a prescribed volume, are required, either for physical consistency or to satisfy engineering specifications. To incorporate this requirement into \eqref{eq:shape_opti}, we impose the volume constraint $|\Omega_0| = \gamma_0$, with $0 < \gamma_0 < |\Omega|$, through a penalty formulation, where deviations from the target volume are penalized quadratically. This transforms the constrained problem into an unconstrained optimization problem with an augmented functional, thereby facilitating the computation of shape gradients and the use of gradient-based algorithms.

Moreover, to enhance numerical stability and computational robustness, the perimeter regularization term $\mathcal{P}_{\Omega}(\Omega_0)$ is replaced by an explicit boundary integral over $\partial \Omega_0$. The resulting formulation accounts for both boundary regularization and volume preservation, and is given by
\begin{equation}
	\min_{\Omega_0 \subset \Omega} J(\Omega_0) = \frac{1}{2} \Vert u_2 \Vert_{0,\Omega}^2 + \frac{1}{2\varepsilon} \Vert w_1 \Vert_{0,\Omega_0}^2+\lambda \int_{\partial \Omega_0}ds + \beta(\int_{\Omega} \chi_{\Omega_0}dx-\gamma_0)^2,
	\label{eq:mini_opti}
\end{equation}
where $\lambda > 0$ and $\beta \geq 0$. To derive the shape derivative of Eqn.\eqref{eq:mini_opti}, we introduce the following associated adjoint system. Similar to the derivation in Section \ref{subsec:CCBM}, we incorporate the weak formulation of the coupled state system \eqref{eq:couple weak} into \eqref{eq:shape optimization problem} to construct the Lagrangian functional:
\begin{equation*}
	\begin{aligned}
		\mathcal{L}_2(\Omega_0, u_1, u_2, &w_1, w_2, v_1, v_2, s_1, s_2) = \frac{1}{2} \Vert u_2 \Vert_{0, \Omega}^2 + \frac{1}{2 \varepsilon} \Vert w_1 \Vert_{0, \Omega_0}^2 \\[0.5em] 
		&+ (D \nabla u_1, \nabla v_1)_{\Omega} + (\mu_a u_1, v_1)_{\Omega} - \alpha (u_2, v_1)_{\Gamma} - \frac{1}{\varepsilon} (w_1, v_1)_{\Omega_0} - (g_1, v_1)_{\Gamma} \\
		&+ (D \nabla u_2, \nabla v_2)_{\Omega} + (\mu_a u_2, v_2)_{\Omega} + \alpha (u_1, v_2)_{\Gamma} - \alpha (g_2, v_2)_{\Gamma} \\[0.5em]
		&+ (D \nabla w_1, \nabla s_1)_{\Omega} + (\mu_a w_1, s_1)_{\Omega} + \alpha (w_2, s_1)_{\Gamma} \\[0.5em]
		&+ (D \nabla w_2, \nabla s_2)_{\Omega} + (\mu_a w_2, s_2)_{\Omega} - \alpha (w_1, s_2)_{\Gamma} + (u_2, s_2)_{\Omega}.
	\end{aligned}
\end{equation*}	
Then, taking the partial derivatives of the Lagrangian with respect to all variables and setting them to zero yields the adjoint system.
\begin{equation*}
	\begin{aligned}
		\frac{\partial \mathcal{L}_2}{\partial u_1} \cdot \nu = 0& \quad \Rightarrow \quad
		(D \nabla v_1, \nabla \nu)_{\Omega} + (\mu_a v_1, \nu)_{\Omega} + \alpha (v_2, \nu)_{\Gamma} = 0 \quad \forall \nu \in H^1(\Omega),\\
		\frac{\partial \mathcal{L}_2}{\partial u_2} \cdot \nu = 0& \quad \Rightarrow \quad
		(D \nabla v_2, \nabla \nu)_{\Omega} + (\mu_a v_2, \nu)_{\Omega} - \alpha (v_1, \nu)_{\Gamma} + (u_2+s_2, \nu)_{\Omega}= 0 \quad \forall \nu \in H^1(\Omega), \\
		\frac{\partial \mathcal{L}_2}{\partial w_1} \cdot \nu = 0& \quad \Rightarrow \quad
		(D \nabla s_1, \nabla \nu)_{\Omega} + (\mu_a s_1, \nu)_{\Omega} - \alpha (s_2, \nu)_{\Gamma} + \frac{1}{\varepsilon} (w_1-v_1, \nu)_{\Omega_0} = 0 \quad \forall \nu \in H^1(\Omega),\\
		\frac{\partial \mathcal{L}_2}{\partial w_2} \cdot \nu = 0& \quad \Rightarrow \quad
		(D \nabla s_2, \nabla \nu)_{\Omega} + (\mu_a s_2, \nu)_{\Omega} + \alpha (s_1, \nu)_{\Gamma} = 0 \quad \forall \nu \in H^1(\Omega),\\
		\frac{\partial \mathcal{L}_2}{\partial v_1} \cdot \nu = 0& \Rightarrow \eqref{eq:lag_1}, \quad
		\frac{\partial \mathcal{L}_2}{\partial v_2} \cdot \nu = 0 \Rightarrow \eqref{eq:lag_2},\quad
		\frac{\partial \mathcal{L}_2}{\partial s_1} \cdot \nu = 0 \Rightarrow \eqref{eq:lag_3}, \quad
		\frac{\partial \mathcal{L}_2}{\partial s_2} \cdot \nu = 0 \Rightarrow \eqref{eq:lag_4}.
	\end{aligned}		
\end{equation*}
The adjoint pair $(v,s) \in \mathbf{H}^1(\Omega) \times \mathbf{H}^1(\Omega)$ satisfies the following coupled boundary value problem:
\begin{equation}
	\left \{
	\begin{aligned}		
		- \textrm{div} (D \nabla v) + \mu_{a} v &= -i(u_2+s_2),\ \ &\textrm{in } \Omega,\\[0.5em]
		D \partial_n v + i \alpha v &= 0,\ \ &\textrm{on } \Gamma,\\[0.5em]
		- \textrm{div} (D \nabla s) + \mu_{a} s &= -\frac{1}{\varepsilon} (w_1-v_1) \chi_{\Omega_0},\ \ &\textrm{in } \Omega,\\[0.5em]
		D \partial_n s + i \alpha s &= 0,\ \ &\textrm{on } \Gamma.
	\end{aligned}
	\right.
	\label{eq:adjoint_mini_opti}
\end{equation}
It can be readily verified that the above problem admits a solution $(v,s) = (w_1, 0)$.

For all $t \in [0,\tau]$ and $\xi_1, \xi_2, \xi_3, \xi_4, \eta_1, \eta_2, \eta_3, \eta_4 \in H^1(\Omega_t)$, we can construct a Lagrangian functional as follows (\cite{Delfour2011}):
\begin{equation}
	\begin{aligned}
		&\mathcal{L}(\Omega_t, \xi_1, \xi_2, \xi_3, \xi_4, \eta_1, \eta_2, \eta_3, \eta_4) \\[1em]
		&= \frac{1}{2}\int_{\Omega_t}\xi_2^2dx + \frac{1}{2\varepsilon}\int_{{\Omega_0}_t}\xi_3^2dx + \lambda \int_{\partial {\Omega_0}_t}ds + \beta(\int_{{\Omega_0}_t}dx-\gamma_0)^2 \\[1em]
		&+ \sum_{i=1}^{4} \int_{\Omega_t}(D \nabla \xi_i \cdot \nabla \eta_i + \mu_{a} \xi_i \eta_i)dx - \int_{\Gamma_t} \alpha \xi_2 \eta_1 ds - \dfrac{1}{\varepsilon} \int_{{\Omega_0}_t} \xi_3 \eta_1 dx - \int_{\Gamma_t} g_1 \eta_1 ds \\[1em]
		&+ \int_{\Gamma_t} \alpha \xi_1 \eta_2 ds - \int_{\Gamma_t} \alpha g_2 \eta_2 ds + \int_{\Gamma_t} \xi_4 \eta_3 ds - \int_{\Gamma_t} \xi_3 \eta_4 ds + \int_{\Omega_t} \xi_2 \eta_4 dx.
	\end{aligned}
	\label{eq:Lagran_func}
\end{equation}
By calculating the first-order optimality conditions of the Lagrangian functional \eqref{eq:Lagran_func}, we obtain
\begin{equation*}
	\begin{aligned}
		\frac{\partial \mathcal{L}(\Omega_t, u_{1t}, u_{2t}, w_{1t}, w_{2t}, v_{1t}, v_{2t}, s_{1t}, s_{2t})}{\partial x_t}[\delta x] = 0, \quad \forall\, x \in \{u_{1},u_{2},w_{1},w_{2},v_{1},v_{2},s_{1},s_{2}\},\;
	\end{aligned}
\end{equation*}
for any $\delta u_1, \delta u_2, \delta w_1, \delta w_2, \delta v_1, \delta v_2, \delta s_1, \delta s_2 \in H^1(\Omega_t)$. From these optimality conditions, we deduce that the objective functional $J(\Omega_t)$ can be expressed as a min-max of the Lagrangian functional $\mathcal{L}$ with the saddle point $(u_{1t}, u_{2t}, w_{1t}, w_{2t}, v_{1t}, v_{2t}, s_{1t}, s_{2t})$, that is,
\begin{equation}
	J(\Omega_t) = \min_{(\xi_1, \xi_2, \xi_3, \xi_4)} \max_{(\eta_1, \eta_2, \eta_3, \eta_4)} \mathcal{L}(\Omega_t, \xi_1, \xi_2, \xi_3, \xi_4, \eta_1, \eta_2, \eta_3, \eta_4).
\end{equation}
\begin{remark}
	C\'ea’s method streamlines the computation of shape derivatives by embedding the state and adjoint equations into a unified Lagrangian framework. By treating the optimization task as a saddle-point problem, the method ensures that the Lagrangian is stationary with respect to all the PDE variables. As a result, only its explicit dependence on geometry must be differentiated, making the derivation of shape sensitivities significantly more direct. The adjoint variables naturally emerge as Lagrange multipliers associated with the PDE constraints.
\end{remark}
To remove the time dependence in the underlying function spaces, we parameterize the functions in $H^1(\Omega_t)$ by elements of $H^1(\Omega)$ through the transformation
\begin{equation*}
	H^1(\Omega_t) = \{\xi \circ T_t^{-1}:\xi \in H^1(\Omega)\}.
\end{equation*}
Using this parametrization, the Lagrange functional $\mathcal{L}$ can be formulated as:
\begin{equation}
	\begin{aligned}
		&\bar{\mathcal{L}}(t, \xi_1, \xi_2, \xi_3, \xi_4, \eta_1, \eta_2, \eta_3, \eta_4) \\[1em]
		&= \mathcal{L}(T_t(\Omega)[\boldsymbol{V}], \xi_1 \circ T_t^{-1}, \xi_2 \circ T_t^{-1}, \xi_3 \circ T_t^{-1}, \xi_4 \circ T_t^{-1}, \eta_1 \circ T_t^{-1}, \eta_2 \circ T_t^{-1}, \eta_3 \circ T_t^{-1}, \eta_4 \circ T_t^{-1}),
	\end{aligned}
\end{equation}
with $\xi_1, \xi_2, \xi_3, \xi_4, \eta_1, \eta_2, \eta_3, \eta_4 \in H^1(\Omega)$. Let $(u_1, u_2, w_1, w_2, v_1, v_2, s_1, s_2)$ be the solutions of the state and adjoint systems. Then, using C\'ea's method, we have
\begin{equation}
	\begin{aligned}
		dJ(\Omega; \boldsymbol{V}) = &\min_{(\xi_1, \xi_2, \xi_3, \xi_4)} \max_{(\eta_1, \eta_2, \eta_3, \eta_4)} \partial_t \bar{\mathcal{L}}(\Omega_t, \xi_1, \xi_2, \xi_3, \xi_4, \eta_1, \eta_2, \eta_3, \eta_4)|_{t=0} \\[1em]
		= &\partial_t \bar{\mathcal{L}}(t,u_1, u_2, w_1, w_2, v_1, v_2, s_1, s_2)|_{t=0},
	\end{aligned}
	\label{cea_method}
\end{equation}
for $\boldsymbol{V} \in \mathcal{U}$ with $\mathcal{U} = \{g \in \mathbf{W}^{1,\infty}(\Omega):g|_{\Gamma} = 0\}$. Formal C\'ea's method is rigorous if we can prove the shape differentiability of the state equation with respect to the domain, as in Proposition \ref{thm:existence of weak material derivative}. Next, we calculated the shape derivative of the objective functional. Since $\boldsymbol{V}|_{\Gamma} = 0$, we rewrite the Lagrangian $\bar{\mathcal{L}}$ defined on the perturbed domain $\Omega_t$ onto the fixed domain $\Omega$:
\begin{equation*}
	\begin{aligned}
		&\bar{\mathcal{L}}(t, \xi_1, \xi_2, \xi_3, \xi_4, \eta_1, \eta_2, \eta_3, \eta_4) \\[1em]
		&= \frac{1}{2}\int_{\Omega_t}(\xi_2 \circ T_t^{-1})^2dx + \frac{1}{2\varepsilon}\int_{{\Omega_0}_t}(\xi_3 \circ T_t^{-1})^2dx + \lambda \int_{\partial {\Omega_0}_t}ds + \beta(\int_{{\Omega_0}_t}dx -\gamma_0) \\[1em]
		&+ \sum_{i=1}^{4} \int_{\Omega_t}(D \nabla (\xi_i \circ T_t^{-1}) \cdot \nabla (\eta_i \circ T_t^{-1}) + \mu_{a} (\xi_i \circ T_t^{-1}) (\eta_i \circ T_t^{-1}))dx \\[1em]
		&- \int_{\Gamma_t} \alpha (\xi_2 \circ T_t^{-1}) (\eta_1 \circ T_t^{-1}) ds - \dfrac{1}{\varepsilon} \int_{{\Omega_0}_t} (\xi_3 \circ T_t^{-1}) (\eta_1 \circ T_t^{-1}) dx - \int_{\Gamma_t} g_1 (\eta_1 \circ T_t^{-1}) ds \\[1em]
		&- \int_{\Gamma_t} \alpha g_2 (\eta_2 \circ T_t^{-1}) ds + \int_{\Gamma_t} \alpha (\xi_1 \circ T_t^{-1}) (\eta_2 \circ T_t^{-1}) ds + \int_{\Gamma_t} (\xi_4 \circ T_t^{-1}) (\eta_3 \circ T_t^{-1}) ds \\[1em]
		&- \int_{\Gamma_t} (\xi_3 \circ T_t^{-1}) (\eta_4 \circ T_t^{-1}) ds + \int_{\Omega_t} (\xi_2 \circ T_t^{-1}) (\eta_4 \circ T_t^{-1}) dx.
	\end{aligned}
\end{equation*}
Then, with a simplification, we obtain
\begin{equation}
	\begin{aligned}
		&\bar{\mathcal{L}}(t, \xi_1, \xi_2, \xi_3, \xi_4, \eta_1, \eta_2, \eta_3, \eta_4) \\[1em]
		&= \frac{1}{2}\int_{\Omega} J(t) \xi_2^2 dx + \frac{1}{2\varepsilon}\int_{\Omega_0} J(t) \xi_3^2 dx + \lambda \int_{\partial \Omega_0} M(t) ds + \beta(\int_{\Omega_0} J(t) dx-\gamma_0) \\[1em]
		&+ \sum_{i=1}^{4} \int_{\Omega} A(t)(D\nabla \xi_i \cdot \nabla \eta_i) dx + \int_{\Omega} J(t) (\mu_{a} \xi_i \eta_i) dx - \int_{\Gamma} \alpha \xi_2 \eta_1 ds - \dfrac{1}{\varepsilon} \int_{\Omega_0} J(t) (\xi_3 \eta_1) dx \\[1em]
		&- \int_{\Gamma} g_1 \eta_1 ds + \int_{\Gamma} \alpha \xi_1 \eta_2 ds - \int_{\Gamma} \alpha g_2 \eta_2 ds + \int_{\Gamma} \xi_4 \eta_3 ds - \int_{\Gamma} \xi_3 \eta_4 ds + \int_{\Omega} J(t) \xi_2 \eta_4 dx,
	\end{aligned}
	\label{eq:L_bar}
\end{equation}
where 
\begin{equation*}
	J(t)=|\textrm{det} \mathcal{D} T_t|, \quad A(t)=J(t)(\mathcal{D} T_t)^{-1}\enspace^{*}(\mathcal{D} T_t)^{-1}, \quad M(t)=J(t)|^{*}(\mathcal{D} T_t)^{-1}\boldsymbol{n}|,
\end{equation*}
and $\mathcal{D} T_t$ denotes the Jacobian matrix of $T_t$, $^{*} \mathcal{D} T_t$ denotes the transpose of $\mathcal{D} T_t$.

By calculating the partial derivatives of Eqn.\eqref{eq:L_bar} with respect to $t$, we obtain
\begin{equation}
	\begin{aligned}
		& \partial_t \bar{\mathcal{L}}(t, \xi_1, \xi_2, \xi_3, \xi_4, \eta_1, \eta_2, \eta_3, \eta_4) \\[1em]
		&= \frac{1}{2}\int_{\Omega} J'(t) \xi_2^2 dx + \frac{1}{2\varepsilon}\int_{\Omega_0} J'(t) \xi_3^2 dx + \lambda \int_{\partial \Omega_0} M'(t) ds + \beta \int_{\Omega_0} J'(t) dx\\[1em]
		&+ \sum_{i=1}^{4} \int_{\Omega} A'(t) (D\nabla \xi_i \cdot \nabla \eta_i) dx + \int_{\Omega} J'(t) (\mu_{a} \xi_i \eta_i) dx - \dfrac{1}{\varepsilon} \int_{\Omega_0} J'(t) (\xi_3 \eta_1) dx + \int_{\Omega} J'(t) \xi_2 \eta_4 dx.
	\end{aligned}
\end{equation}
Moreover, the following identities are well known in shape calculus (\cite{Sokolowski1992})
\begin{equation*}
	\begin{aligned}
		A'(0) &= \lim_{t \to 0^+} \frac{A(t)-I}{t} = (\textrm{div}\boldsymbol{V})I - ^{*}\mathcal{D} \boldsymbol{V} - \mathcal{D} \boldsymbol{V} \triangleq \mathcal{E}(\boldsymbol{V}), \\[1em]
		J'(0) &= \lim_{t \to 0^+} \frac{J(t)-1}{t} = \textrm{div}\boldsymbol{V}, \\[1em]
		M'(0) &= \lim_{t \to 0^+} \frac{M(t)-1}{t} = \textrm{div}\boldsymbol{V}  - \mathcal{D} \boldsymbol{V} \boldsymbol{n} \cdot \boldsymbol{n}. \\[1em]
	\end{aligned}
\end{equation*}
Combing these identities with (\ref{cea_method}), we obtain
\begin{equation}
	\begin{aligned}
		dJ(&\Omega; \boldsymbol{V}) = \partial_t \bar{\mathcal{L}}(t,u_1, u_2, w_1, w_2, v_1, v_2, s_1, s_2)|_{t=0} \\[1em]
		&= \frac{1}{2}\int_{\Omega} (\textrm{div}\boldsymbol{V}) u_2^2 dx + \frac{1}{2\varepsilon}\int_{\Omega_0} (\textrm{div}\boldsymbol{V}) w_1^2 dx + \lambda \int_{\partial \Omega_0} (\textrm{div}\boldsymbol{V}  - \mathcal{D} \boldsymbol{V} \boldsymbol{n} \cdot \boldsymbol{n}) ds + \beta \int_{\Omega_0} \textrm{div}\boldsymbol{V} dx\\[1em]
		&+ \int_{\Omega} \mathcal{E}(\boldsymbol{V}) (D\nabla u_1 \cdot \nabla v_1) dx + \int_{\Omega} (\textrm{div}\boldsymbol{V}) (\mu_{a} u_1 v_1) dx - \dfrac{1}{\varepsilon} \int_{\Omega_0} (\textrm{div}\boldsymbol{V}) (w_1 v_1) dx \\[1em]
		&+ \int_{\Omega} \mathcal{E}(\boldsymbol{V}) (D \nabla u_2 \cdot \nabla v_2) dx + \int_{\Omega} (\textrm{div}\boldsymbol{V}) (\mu_{a} u_2 v_2) dx \\[1em]
		&+ \int_{\Omega} \mathcal{E}(\boldsymbol{V}) (D \nabla w_1 \cdot \nabla s_1) dx + \int_{\Omega} (\textrm{div}\boldsymbol{V}) (\mu_{a} w_1 s_1) dx \\[1em]
		&+ \int_{\Omega} \mathcal{E}(\boldsymbol{V}) (D \nabla w_2 \cdot \nabla s_2) dx + \int_{\Omega} (\textrm{div}\boldsymbol{V}) (\mu_{a} w_2 s_2) dx + \int_{\Omega} (\textrm{div}\boldsymbol{V}) (u_2 s_2) dx.
	\end{aligned}	
	\label{shape_derivative_volume_type}
\end{equation}
This yields a distributed form of shape derivative.

\begin{remark}
	Actually, shape derivatives can be represented in two equivalent forms: the distributed form and the surface form. Both forms explicitly involve the perturbation field $\boldsymbol{V}$ without requiring a PDE to be solved for $\boldsymbol{V}$. The distributed form, expressed as a domain integral over $\Omega$, depends on the state variables $u_1, u_2$, the adjoint variables $w_1, w_2$, and the perturbation $\boldsymbol{V}$. In contrast, the surface form requires a higher regularity of the state and adjoint variables to ensure well-defined boundary traces. In practical applications, the distributed form is often preferred due to its improved numerical stability and convenience for theoretical analysis. For these reasons, the analysis in this work is restricted to the distributed form.
\end{remark}

\section{Algorithm}\label{sec:algorithm}
In this section, we develop a two-step algorithm for reconstructing the bioluminescent source. In the first step, the shape of the source region is updated using the level set method, which provides a flexible framework capable of capturing complex geometric variations and topological changes. In the second step, the source intensity is reconstructed via a parameter-dependent CCBM algorithm. This sequential strategy effectively decouples geometric and intensity reconstruction processes, thereby enhancing both the stability and accuracy of the overall inversion.

\subsection{Level set based approach for shape optimization}
To solve the resulting shape optimization problem, we employ the shape gradient descent method. This approach is widely used in shape optimization because of its simplicity and ease of integration with various domain representation and tracking techniques. A crucial component of this approach is determining an appropriate descent direction, where the shape derivative serves as the guiding quantity.

In this work, we adopt the Hilbertian regularization method (\cite{Allaire2021}) to extend and smooth the shape gradient obtained in \eqref{shape_derivative_volume_type}, thereby ensuring numerical stability during the iterative process. The smoothed descent direction is then determined using the standard $H^1$-gradient flow formulation, in which we seek a velocity field $\boldsymbol{V} \in \mathbf{H}_0^1(\Omega)$ that satisfies
\begin{equation}
	(\nabla \boldsymbol{V}, \nabla \boldsymbol{W})+(\boldsymbol{V},\boldsymbol{W}) = -dJ(\Omega;\boldsymbol{W}) \qquad \forall \boldsymbol{W} \in \mathbf{H}_0^1(\Omega).
	\label{eq:hilbert_regular}
\end{equation}

We adopt the level set method to implicitly represent and evolve the domain boundary to efficiently handle large deformations and possible topological changes during the optimization process (\cite{Hu2022}). This approach introduces a level set function whose negative subdomain defines the current shape, thus providing a flexible framework for tracking complex boundary evolutions. For a given domain $ \Omega_0 \subset \Omega $ at a fixed time $ t $, the level set function $ \varphi : [0, \infty) \times \mathbb{R}^{d+1} \to \mathbb{R} $ is defined as:
\begin{equation*}
	\left \{
	\begin{aligned}
		\varphi(t,x)<0 \quad &\textrm{if} \ x \in \Omega_0,\\[0.5em]
		\varphi(t,x)=0 \quad &\textrm{if} \ x \in \partial \Omega_0,\\[0.5em]
		\varphi(t,x)>0 \quad &\textrm{if} \ x \in \Omega \backslash \bar{\Omega}_0.
	\end{aligned}	
	\right.
\end{equation*}
The domain $\Omega_0$ is represented by the negative region of $\varphi$. To update the level set during evolution, consider a point $x(t)$ moving along the flow induced by the velocity field $\boldsymbol{V}$, that is, $\tfrac{dx}{dt} = \boldsymbol{V}(t, x(t))$. By applying the chain rule, we have
\begin{equation*}
	\frac{d}{dt}\varphi(t, x(t)) = \partial_t \varphi(t,x) + \boldsymbol{V}(t,x) \cdot \nabla \varphi(t,x).
\end{equation*}
We then set this total derivative to zero, thereby ensuring that $\varphi(t, x(t))$ remains constant along each trajectory of the velocity field. In particular, the zero level set of $\varphi$, which represents the domain boundary, moves consistently with the velocity field $\boldsymbol{V}$. Consequently, the evolution of the level set function is governed by the following advection-type equation (\cite{Osher1988, Zhu2018}):
\begin{equation}
	\left \{
	\begin{aligned}
		\partial_t \varphi(t,x) + \boldsymbol{V}(t,x) \cdot \nabla \varphi (t,x) = 0 \quad &\textrm{in} \  (0,T) \times \Omega, \\[0.5em]
		\varphi (0,x) = \varphi_0 \quad &\textrm{in} \  \Omega, \\[0.5em]
		\partial_n \varphi(t,x) = 0 \quad &\textrm{on} \ \Gamma.
	\end{aligned}
	\label{eq:level_set_update}
	\right.
\end{equation}
During the iteration, it is important to ensure that the level set function $\varphi$ is neither too flat nor too steep near the boundary, to maintain the numerical stability and accuracy. Notably, the signed distance function $ d(x, \partial \Omega) := \min_{p \in \partial \Omega} |x - p| $ naturally satisfies this requirement. Therefore, the level set function can be periodically reinitialized as a signed distance function, in order to prevent degradation during shape evolution. In this work, we employ the redistancing equation to ensure the numerical stability of the level set method.
\begin{equation}
	\left \{
	\begin{aligned}
		\partial_t \varphi(t,x) + \textrm{sgn} (\varphi) (|\nabla \varphi (t,x)|-1) = 0 \quad &\textrm{in} \ (0,T) \times \Omega,\\[0.5em]
		\varphi (0,x) = \varphi_0 \quad &\textrm{in} \  \Omega.
	\end{aligned}
	\label{eq:level_set_reinitial}
	\right.
\end{equation}
\begin{remark}
	In the context of shape and topology optimization, the time variable in \eqref{eq:level_set_update} and \eqref{eq:level_set_reinitial} is not the “true” time but rather some pseudo-time corresponding to the descent parameter in the minimization of the objective function. Likewise, the time step in the discretization of \eqref{eq:level_set_update} and \eqref{eq:level_set_reinitial} is a descent step.
\end{remark}
Combining the level set representation of domains with the steepest method, we present the shape optimization algorithm as Algorithm \ref{algorithm}.

\begin{algorithm} \label{algorithm}
	\caption{Shape Steepest Descent Method}
	\label{alg:SSDM}
	\renewcommand{\algorithmicrequire}{\textbf{Input:}}
	\renewcommand{\algorithmicensure}{\textbf{Output:}}
	
	\begin{algorithmic}[1]
		\REQUIRE $g_1, g_2, \varepsilon, \beta, \lambda, \Omega_{init}$ 		%%input
		\ENSURE $\varphi, \Omega_0$    %%output
		
		\STATE  $k=0$
		
		\WHILE{$|J(\Omega^{k+1})-J(\Omega^k)| > \tilde{\varepsilon} |J(\Omega^k)|$}
		\STATE solve the state system \eqref{eq:couple weak} by finite element method.
		\STATE obtain the solution of the adjoint system \eqref{eq:adjoint_mini_opti} through the state solution.
		\STATE compute the descent direction by solving the Hilbertian regularization equation \eqref{eq:hilbert_regular}.
		\STATE update the level set function by solving \eqref{eq:level_set_update}.
		\STATE check if reinitialization is needed; if so, reinitialize by \eqref{eq:level_set_reinitial}.
		\STATE $k=k+1$
		\ENDWHILE
		
		\RETURN Outputs
	\end{algorithmic}
\end{algorithm}

\begin{remark}
	We emphasize that our approach differs completely from that presented in \cite{Ding2024}. In the cited work, the authors employed parametric regularization to address the ill-posedness of the problem, utilizing a parametric model of shape variations for numerical implementation. This process involved restricting the domains to simple regular geometries and using Cartesian coordinates for parametrization.	
\end{remark}

\subsection{Reconstruction of source intensity}
We remark that our proposed algorithm can simultaneously recover the support and intensity of the source by setting $\phi(x) = \tfrac{1}{\varepsilon}w_1(x)|_{\Omega_0}$. After completing the shape optimization, a standard inverse source method can be applied to refine the intensity $\phi(x)$ using the recovered support $\Omega_0$. By setting $\alpha = \sqrt{\varepsilon}$, we solve the following linear system:
\begin{equation}
	\left \{
	\begin{aligned}
		&(D \nabla u_1, \nabla v_1)_{\Omega} + (\mu_{a} u_1,v_1)_{\Omega} - \alpha (u_2,v_1)_{\Gamma} - \tfrac{1}{\varepsilon} (w_1, v_1) = (g_1,v_1)_{\Gamma},&\forall \ v_1 \in H^1(\Omega),\\[0.5em]
		&(D \nabla u_2, \nabla v_2)_{\Omega} + (\mu_{a} u_2,v_2)_{\Omega} + \alpha(u_1,v_2)_{\Gamma} = \alpha (g_2,v_2)_{\Gamma}, &\forall \ v_2 \in H^1(\Omega),\\[0.5em]
		&(D \nabla w_1, \nabla s_1)_{\Omega} + (\mu_{a} w_1,s_1)_{\Omega} - \alpha (w_2,s_1)_{\Gamma} - (u_2,s_2)_{\Omega}= 0, &\forall \ s_1 \in H^1(\Omega),\\[0.5em]
		&(D \nabla w_2, \nabla s_2)_{\Omega} + (\mu_{a} w_2,s_2)_{\Omega} + \alpha (w_1,s_2)_{\Gamma} = 0, &\forall \ s_2 \in H^1(\Omega).
	\end{aligned}
	\right . \label{eq:ccbm_intensity}
\end{equation}
Then, by substituting the computed $w_1$ into $\phi(x) = \tfrac{1}{\varepsilon} w_1(x) \chi_{\Omega_0}$, we obtain the reconstructed source intensity $\phi$.

\begin{remark}
	Our algorithm differs from that in \cite{GongWei2025} by employing a two-step strategy to reconstruct the source region and intensity separately, with both steps incorporating the parameter-dependent CCBM method. This sequential approach yields a  significantly more accurate source intensity, as demonstrated by the numerical results in the following section.
\end{remark}

\section{Numerical results} \label{sec:experiment}
In this section, we present numerical experiments to demonstrate the effectiveness and efficiency of the proposed algorithm in a two-dimensional case. The computational domain is fixed as $\Omega = { (x,y) \in \mathbb{R}^2 : |x|<1, |y|<1}$, and all experiments are conducted using the open-source software, NGSolve. Data $g_2$ is generated from the solution of Eqn.\eqref{eq:bvp_reduced} using $g_1$ and the exact source $\phi_*$ on a fine finite element mesh with 41,032 elements and 20,783 vertices. The shape optimization problem is then solved on a coarser mesh with 5,820 elements and 3,011 vertices, where the domain is discretized using Delaunay triangulation and Algorithm \ref{alg:SSDM} is employed to compute the approximate domain $\Omega_\varepsilon$.

The regularization parameter $\varepsilon$, the penalty parameter $\lambda$ for the perimeter constraint, and the penalty parameter $\beta$ for the volume constraint are kept fixed during the first 20 iterations and subsequently reduced with a decay rate of $0.9$. The choice of the regularization parameter $\varepsilon$ plays a crucial role in the regularization strategy. Following Proposition \ref{prop:select_alpha}, we set $\alpha = c \sqrt{\varepsilon}$ with $c = 10^2$.

The noisy measured data are generated as
\begin{equation*}
	g_2^{\delta} = g_2 + \delta g_2(2 {\rm rand} ({\rm size}(g_2))-1),
\end{equation*}
where $\delta$ denotes the noise level, and ${\rm rand}$ represents uniformly distributed random numbers in the interval $[0,1]$.

We define the area error and the intensity error as
\begin{equation*}
	{\rm err} (\Omega_{\varepsilon}) = 1 - \frac{2|\Omega_{\varepsilon} \cap \Omega_*|}{|\Omega_\varepsilon| + |\Omega_*|}, \qquad
	{\rm err} (\phi_\varepsilon) = \left( \int_{\Omega_*} (\phi_\varepsilon-\phi_*)^2 dx \right)^{\frac{1}{2}} / \left( \int_{\Omega_*} \phi_*^2 dx \right)^{\frac{1}{2}},
\end{equation*}
which measure the relative error in the support region and the source intensity, respectively. Here, the subscript $*$ denotes the true value, the subscript $\varepsilon$ denotes the reconstructed value, and $|\cdot|$ denotes the area of a region.

For comparison, we also implemented the method proposed in \cite{GongWei2025} (hereafter referred to as “GZ”). Under identical conditions, we compared the best performances of the two algorithms. Without loss of generality, we set $D = 1$, $\mu_a = 1$, and $A = 0.5$, which is consistent with GZ, and $g_1= D \sin(\pi x) \sin(\pi y)$ on $\Gamma$.

\begin{example}\label{ex:moon}
	In the first example, we take the exact source function $\phi_*=1$. The true domain is given by $\Omega_* = \{(x,y) \in \Omega : 10(x+0.4-y^2)^2+x^2+y^2 <0.5\}$. The initial guess is chosen as $\Omega_{init} = \{(x,y) \in \Omega : x^2+y^2 < 0.04\}$. 
\end{example}

\begin{figure}[htbp]
	\centering
	\subfigure[true source $\Omega_*$ ]{\includegraphics[width=0.3\textwidth]
		{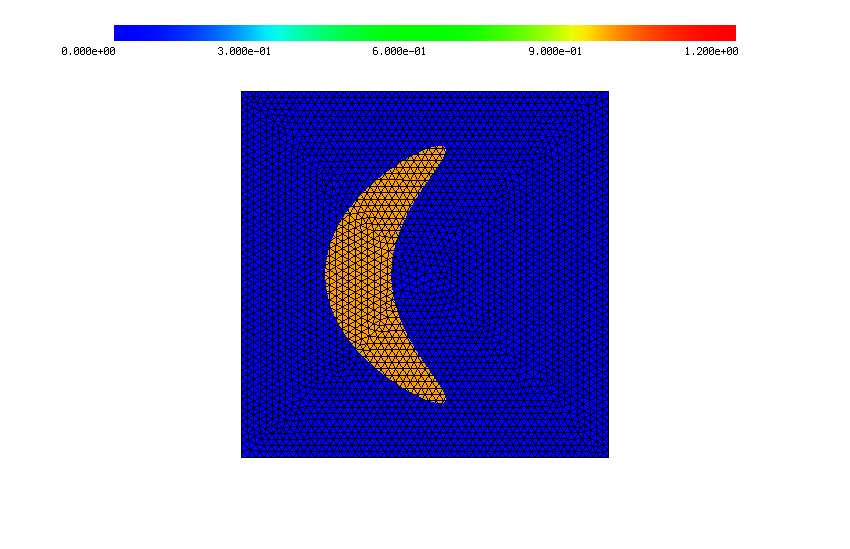}}
	\subfigure[initial domain $\Omega_{init}$]{\includegraphics[width=0.3\textwidth]
		{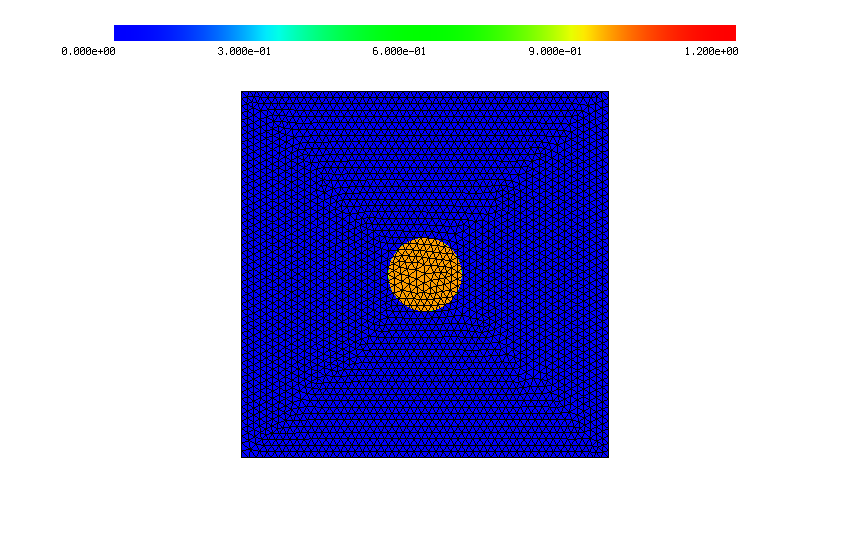}}
	\subfigure[reconstruct without noise]{\includegraphics[width=0.3\textwidth]
		{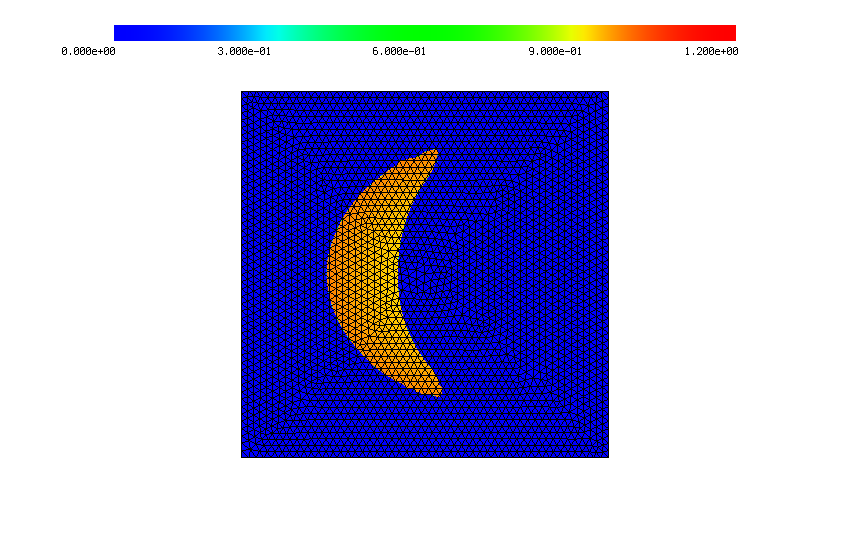}}\\
	\subfigure[reconstruct with 0.01$\%$ noise]{\includegraphics[width=0.3\textwidth]
		{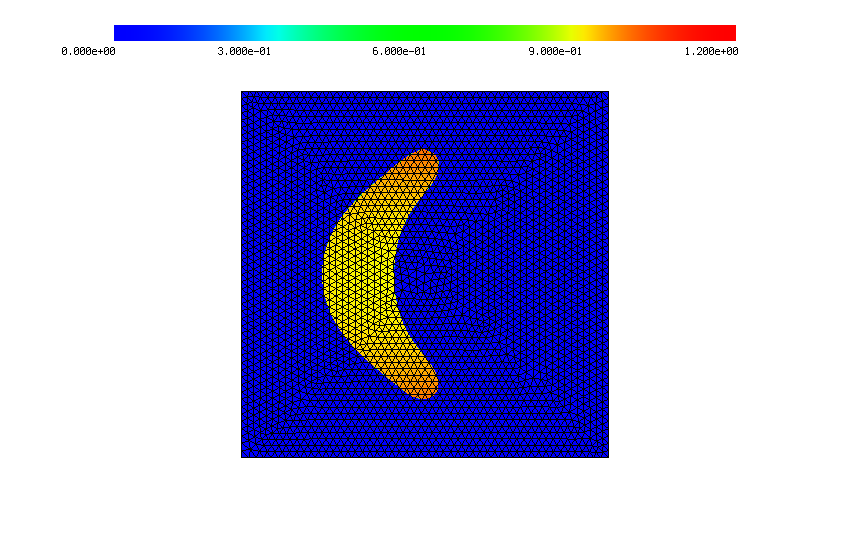}}	
	\subfigure[reconstruct with 0.1$\%$ noise]{\includegraphics[width=0.3\textwidth]
		{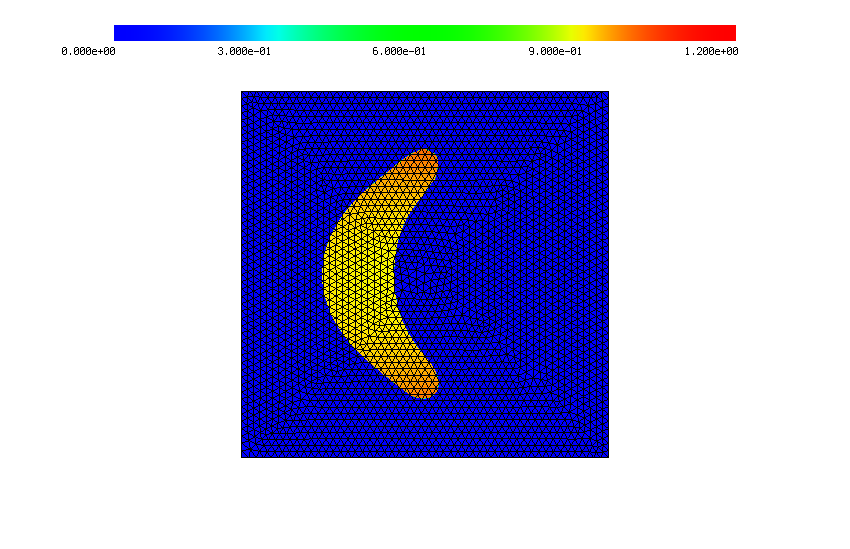}}
	\subfigure[reconstruct with 0.5$\%$ noise]{\includegraphics[width=0.3\textwidth]
		{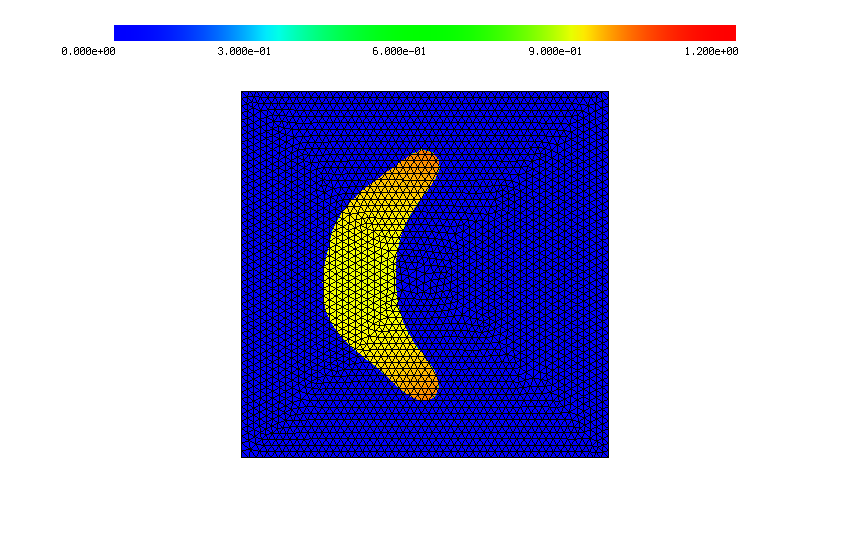}}\\
	\subfigure[reconstruct with 0.01$\%$ noise]{\includegraphics[width=0.3\textwidth]
		{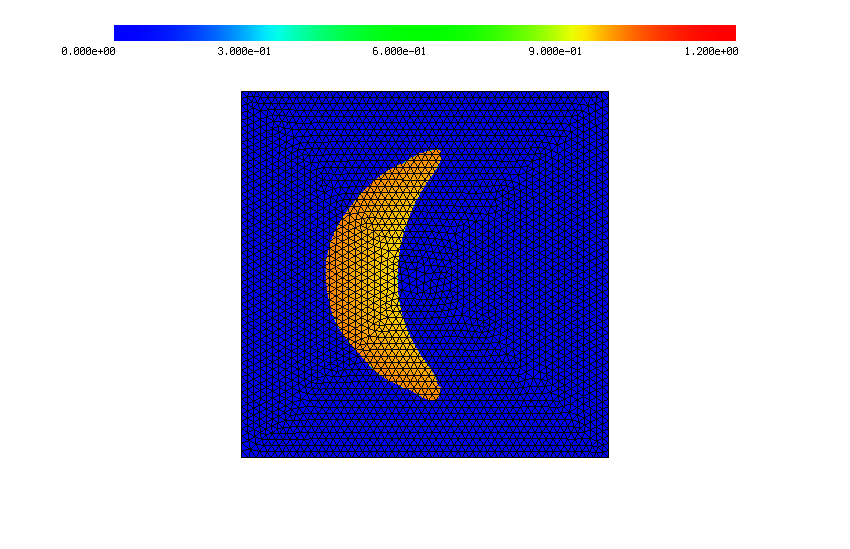}}	
	\subfigure[reconstruct with 0.1$\%$ noise]{\includegraphics[width=0.3\textwidth]
		{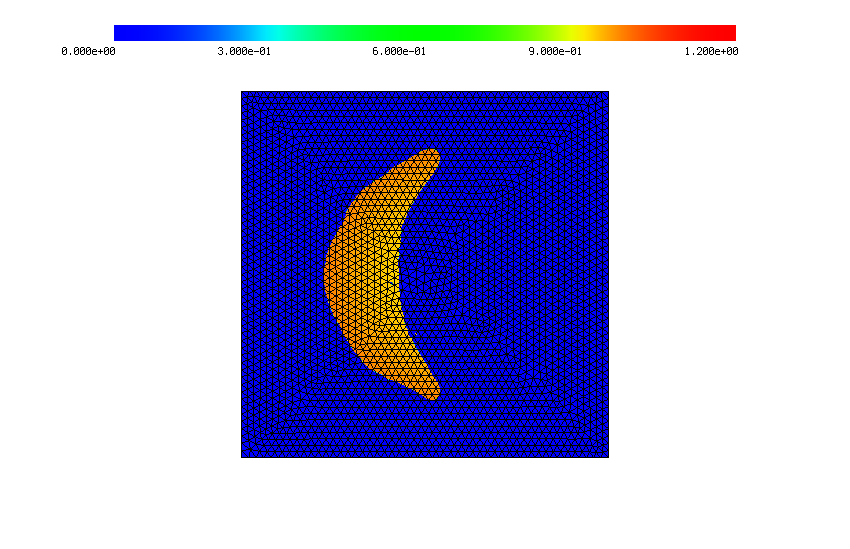}}
	\subfigure[reconstruct with 0.5$\%$ noise]{\includegraphics[width=0.3\textwidth]
		{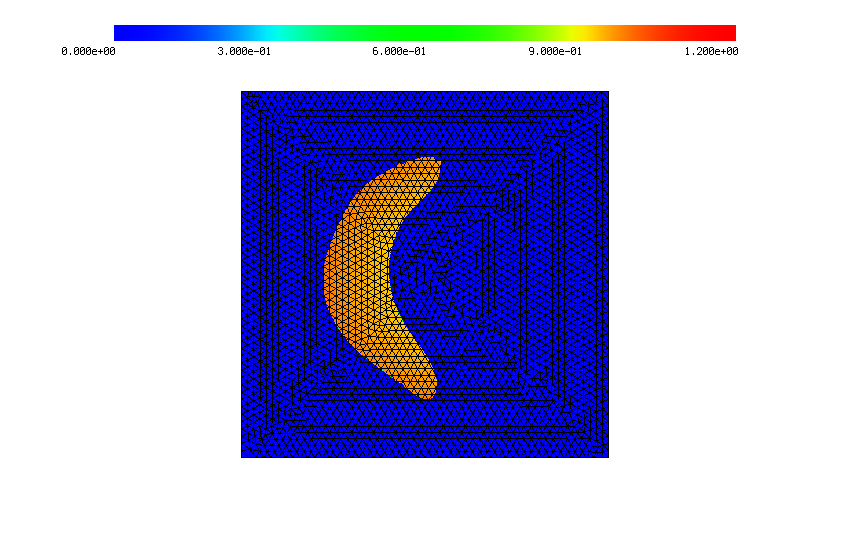}}\\
	\caption{Numerical results for Example \ref{ex:moon} under different noise levels: (d)–(f) GZ method; (g)–(i) our method. The colorbars are uniformly set to [0, 1.2].}
	\label{fig:moon}
\end{figure}

\begin{table}[htbp]
	\caption{Numerical results for Example \ref{ex:moon} with different noise levels.}
	\centering
	\begin{tabular}{lcccc}
		\hline
		Method & Error & $\delta=0.0001$ & $\delta=0.001$ & $\delta=0.005$ \\
		\hline
		GZ & ${\rm err} (\Omega_{\varepsilon})$ & $5.3084 \times 10^{-2}$ & $7.0118 \times 10^{-2}$ & $8.7076 \times 10^{-2}$ \\
		& ${\rm err} (\phi_\varepsilon)$ & $6.4785 \times 10^{-2}$ & $7.5084 \times 10^{-2}$ & $8.5029 \times 10^{-2}$ \\
		\hline
		Ours & ${\rm err} (\Omega_{\varepsilon})$ & $4.1784 \times 10^{-2}$ & $5.2254 \times 10^{-2}$ & $6.1448 \times 10^{-2}$ \\
		& ${\rm err} (\phi_\varepsilon)$ & $2.4354 \times 10^{-2}$ & $2.4357 \times 10^{-2}$ & $2.4360 \times 10^{-2}$ \\
		\hline
	\end{tabular}
	\label{tab:moon}
\end{table}

Table \ref{tab:moon} summarizes the area and intensity errors of the reconstructed source relative to the exact one, obtained by both GZ and our method. Figures \ref{fig:moon} (a) and (c) illustrate that the numerical reconstructions closely approximate the true solution, accurately recovering both the support and intensity of the source.

To evaluate stability under measurement noise, we added uniformly distributed random perturbations to the boundary data $g_2$ with noise levels $\delta = 0.0001, 0.001, 0.005$. The reconstructed domains $\Omega_\varepsilon$ obtained by our method are shown in Figures \ref{fig:moon} (g)–(i), demonstrating high accuracy. For comparison, Figures \ref{fig:moon} (d)–(f) present the results of GZ. Under identical mesh resolution and noise levels, our method yields a sharper reconstruction of the source boundaries, highlighting its superior geometric resolution.

As the noise level increases, the reconstruction error of the source domain also increases. From Table \ref{tab:moon}, the area error of our reconstructions is consistently smaller than that of GZ. Since the accuracy of the domain reconstruction directly affects the intensity reconstruction, the intensity errors obtained by our method are also lower. Owing to the stability and robustness of the parameter-dependent CCBM algorithm employed in computing the source intensity, the increase in intensity error with increasing noise remains moderate.

\begin{example}\label{ex:rect}	
	In the second example, we consider a polygonal domain $\Omega_0$. We set $f=1 \ \textrm{in} \ \Omega$, and the exact source function $\phi_*= 1$. The exact domain $\Omega_*$ is defined as $\{(x,y) \in \Omega : -0.1<x<0.6, 0.1<y<0.4 \}$. We choose the initial domain as $\Omega_{init} = \{(x,y) \in \Omega : x^2+y^2 < 0.04\}$. 
\end{example}

In Table \ref{tab:rect}, we report the relative reconstruction errors, and Figure \ref{fig:rect} displays the recovered support and intensity. As shown in Figures \ref{fig:rect} (a) and (c), the numerical solution matches the true source well, successfully recovering both support and intensity. When uniformly distributed random noise with levels $\delta = 0.0001, 0.001, 0.005$ is added to the boundary data $g_2$, the reconstructions in Figures \ref{fig:rect} (g)–(i) remain stable. For comparison, Figures \ref{fig:rect} (d)–(f) present the results reproduced from GZ. Under identical discretization and noise levels, our method achieves an accuracy comparable to that of GZ.

For polygonal domains, reconstruction algorithms typically experience reduced accuracy near corners, where geometric singularities cause smoothing along the sharp edges. This behavior is also observed in our results: as shown in Figure \ref{fig:rect}, the smooth portions of the domain are accurately recovered, whereas the sharp vertices are only partially reconstructed. Compared with GZ, our method exhibits a similar overall reconstruction quality and slightly improved recovery of corner features.

\begin{figure}[htbp]
	\centering
	\subfigure[true source $\Omega_*$]{\includegraphics[width=0.3\textwidth]		{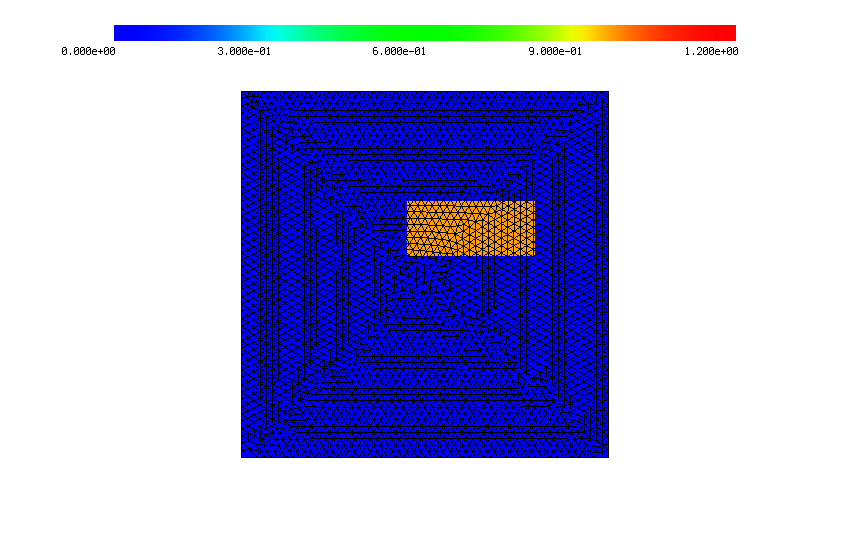}}
	\subfigure[initial domain $\Omega_{init}$]{\includegraphics[width=0.3\textwidth]
		{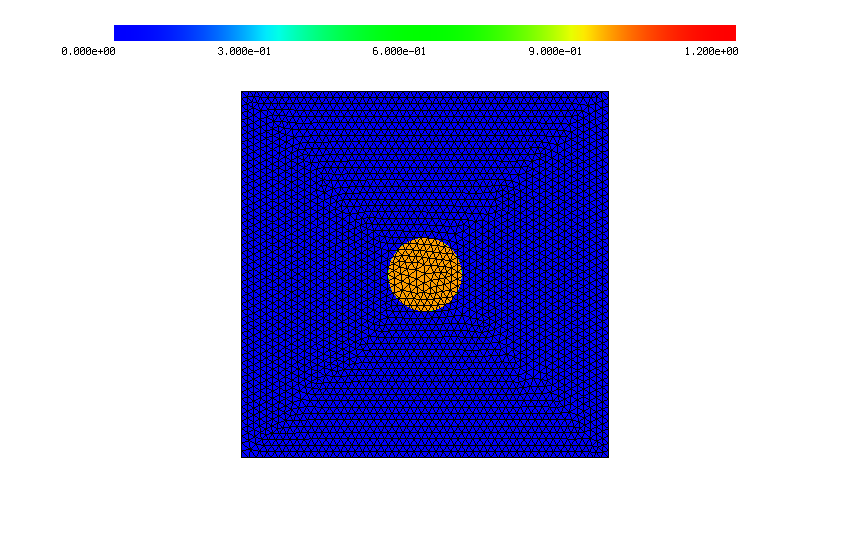}}
	\subfigure[reconstruct without noise]{\includegraphics[width=0.3\textwidth]
		{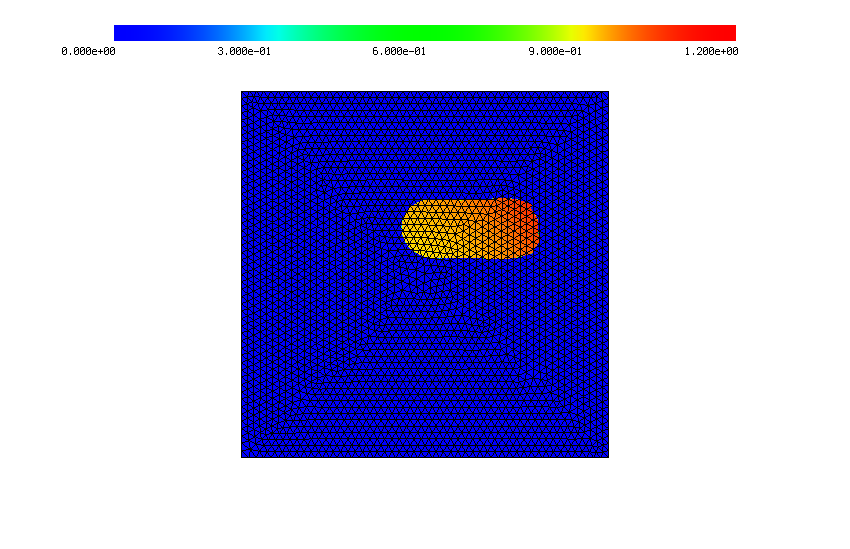}} \\
	\subfigure[reconstruct with 0.01$\%$ noise]{\includegraphics[width=0.3\textwidth]
		{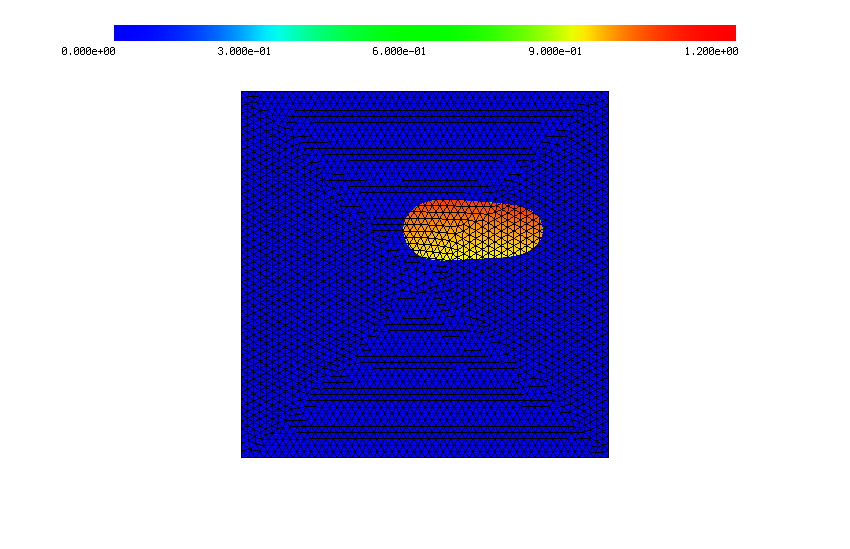}}	
	\subfigure[reconstruct with 0.1$\%$ noise]{\includegraphics[width=0.3\textwidth]
		{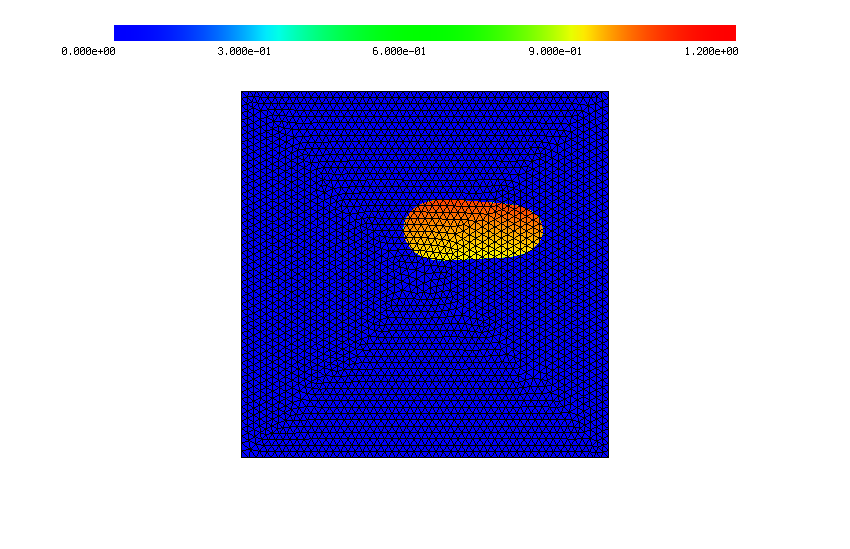}}
	\subfigure[reconstruct with 0.5$\%$ noise]{\includegraphics[width=0.3\textwidth]
		{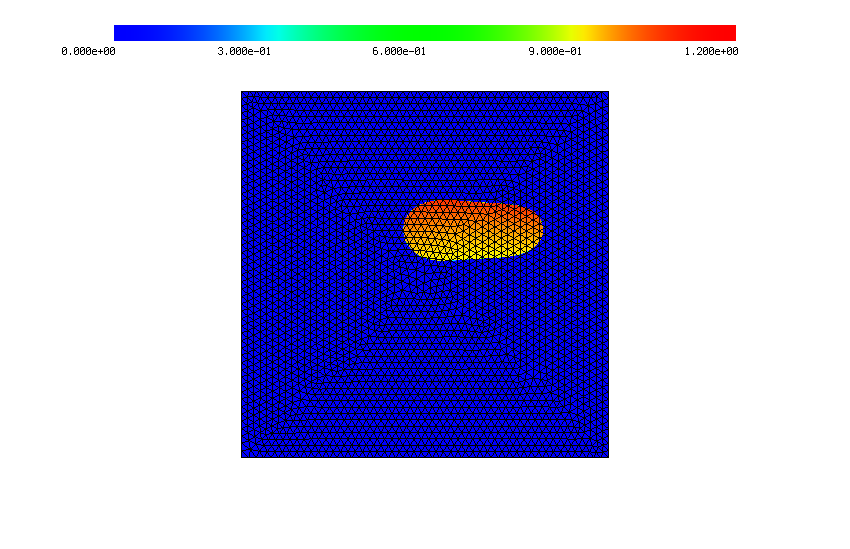}}	\\
	\subfigure[reconstruct with 0.01$\%$ noise]{\includegraphics[width=0.3\textwidth]
		{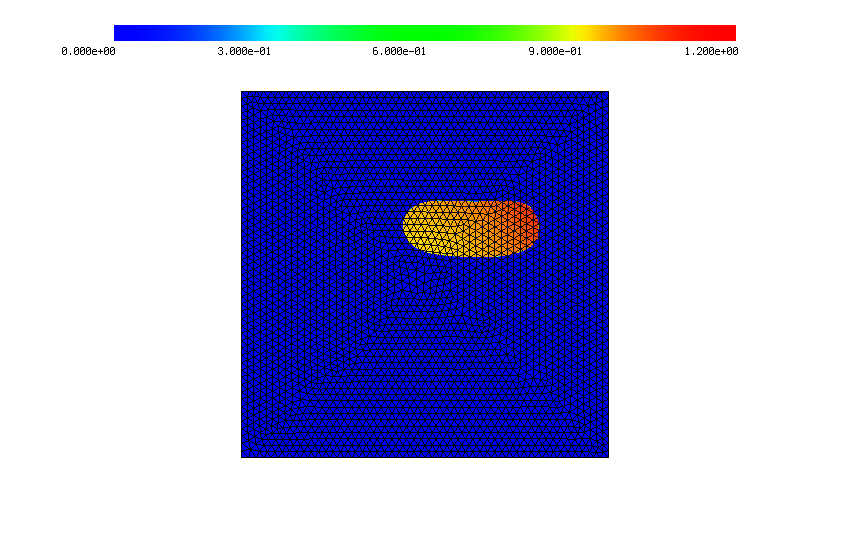}}	
	\subfigure[reconstruct with 0.1$\%$ noise]{\includegraphics[width=0.3\textwidth]
		{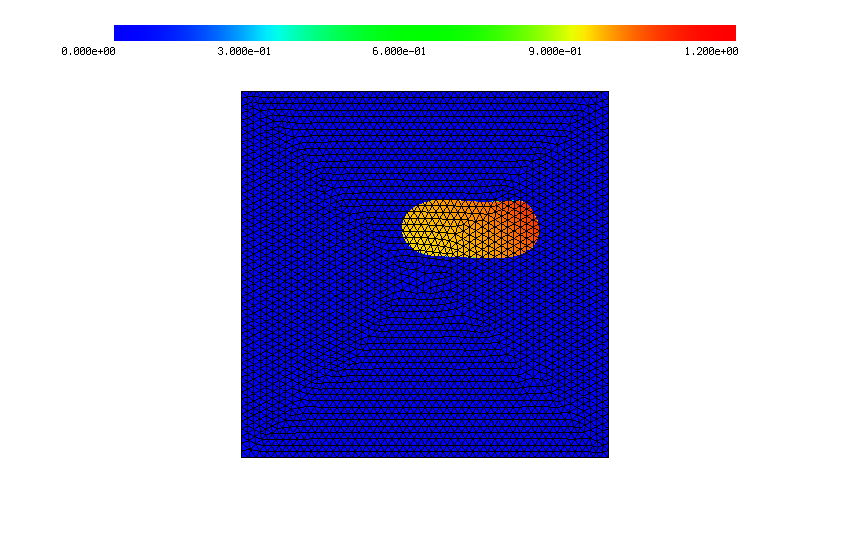}}
	\subfigure[reconstruct with 0.5$\%$ noise]{\includegraphics[width=0.3\textwidth]
		{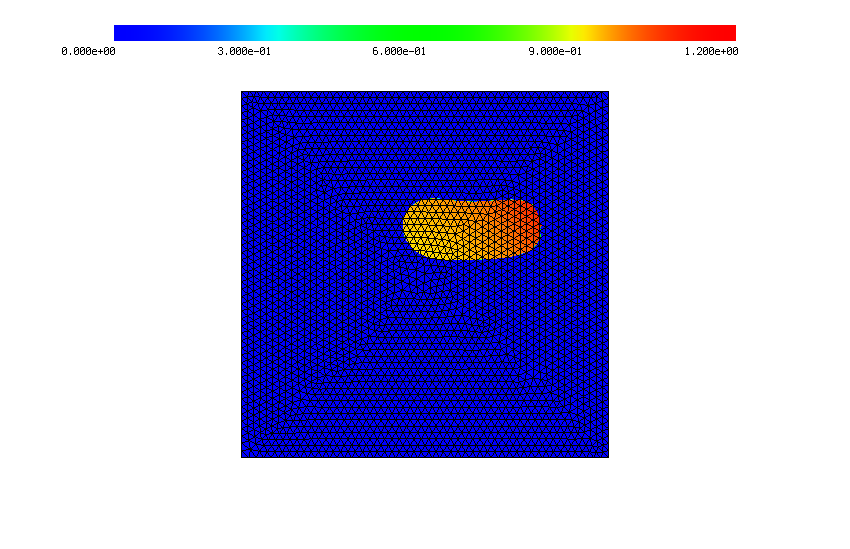}}		
	\caption{Numerical results for Example \ref{ex:rect} under different noise levels: (d)–(f) GZ method; (g)–(i) our method. The colorbars are uniformly set to [0, 1.2].}
	\label{fig:rect}
\end{figure}

\begin{table}[htbp]
	\caption{Numerical results for Example \ref{ex:rect} with different noise levels.}
	\centering
	\begin{tabular}{lcccc}
		\hline
		Method & Error & $\delta=0.0001$ & $\delta=0.001$ & $\delta=0.005$ \\
		\hline
		GZ & ${\rm err} (\Omega_{\varepsilon})$ & $7.2185 \times 10^{-2}$ & $7.4770 \times 10^{-2}$ & $7.5563 \times 10^{-2}$ \\
		& ${\rm err} (\phi_\varepsilon)$ & $7.1015 \times 10^{-2}$ & $7.1156 \times 10^{-1}$ & $7.3875 \times 10^{-1}$ \\
		\hline
		Ours & ${\rm err} (\Omega_{\varepsilon})$ & $4.7587 \times 10^{-2}$ & $5.5005 \times 10^{-2}$ & $6.1843 \times 10^{-2}$ \\
		& ${\rm err} (\phi_\varepsilon)$ & $5.1645 \times 10^{-2}$ & $5.1644 \times 10^{-2}$ & $5.1646 \times 10^{-2}$ \\
		\hline
	\end{tabular}
	\label{tab:rect}
\end{table}

\begin{example}\label{ex:sepa}
	In this example, we set $f=1 \ \textrm{in} \ \Omega$, and define the exact source function $\phi_* = 1$ as well as the exact domain $\Omega_* = \{ (x,y) \in \Omega : (x \pm 0.45)^2 + (y \pm 0.45)^2 < 0.04 \}$. We choose the initial domain as $\Omega_{init} = \{ (x,y) \in \Omega : x^2 + y^2 < 0.07 \}$. 
\end{example}

Table \ref{tab:sepa} reports the relative errors, and Figure \ref{fig:sepa} shows the reconstructed support and intensity. This example demonstrates the ability of our algorithm to handle topological changes, as the true source consists of two separated circular subdomains. As shown in Figures \ref{fig:sepa} (a) and (c), the numerical reconstruction agrees well with the true solution, accurately recovering both the support and the intensity. When uniformly distributed random noise with levels $\delta = 0.0001, 0.001, 0.005$ is added to the boundary measurement $g_2$, the reconstructions in Figures \ref{fig:sepa} (g)–(i) remain satisfactory. For comparison, Figures \ref{fig:sepa} (d)–(f) present the results reproduced using the GZ method. Under the same mesh resolution and noise levels, the GZ reconstruction tends to deviate from circularity after the two components separate, exhibiting noticeable droplet-like distortions, whereas our method preserves the circular shape of both subdomains with significantly higher fidelity.

\begin{figure}[htbp]
	\centering
	\subfigure[true source $\Omega_*$ ]
	{\includegraphics[width=0.3\textwidth]
		{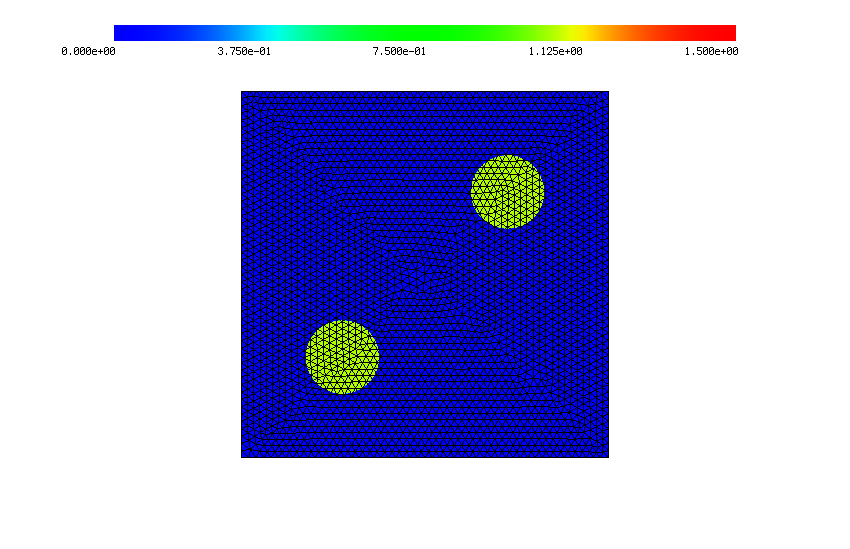}}
	\subfigure[initial domain $\Omega_{init}$]
	{\includegraphics[width=0.3\textwidth]
		{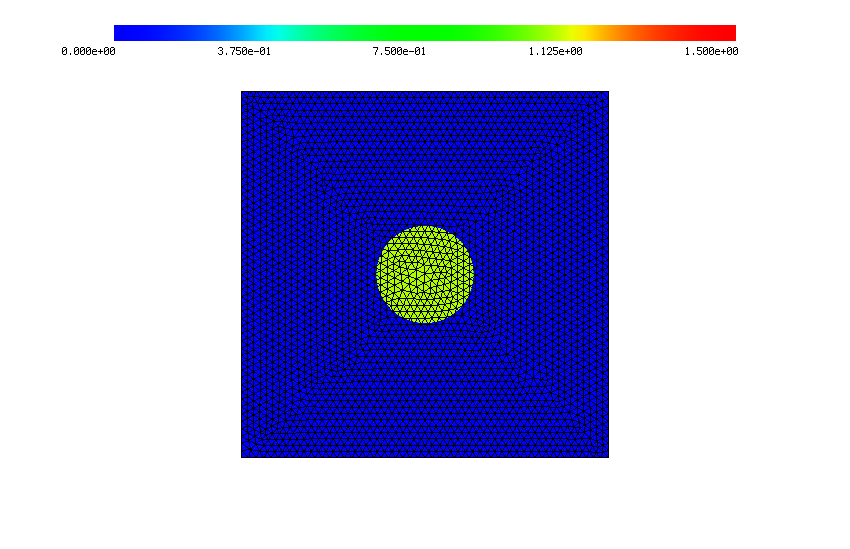}}
	\subfigure[reconstruct without noise]{\includegraphics[width=0.3\textwidth]
		{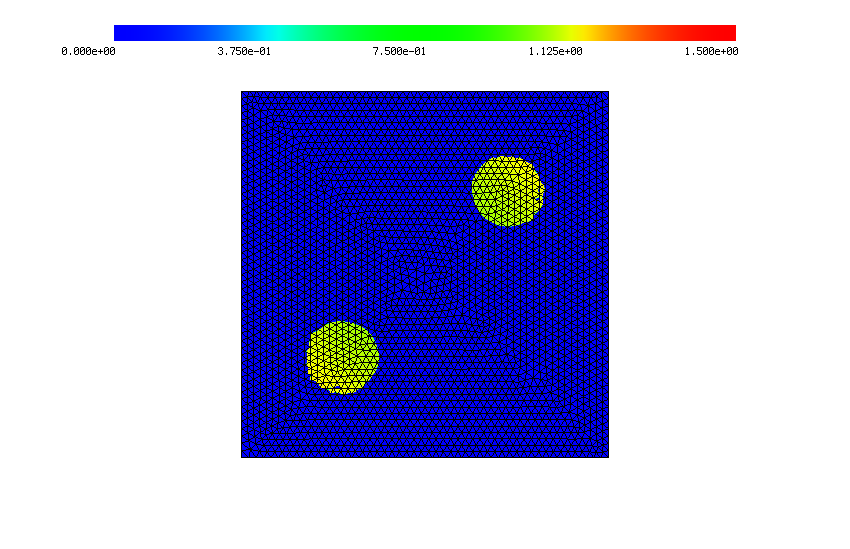}} \\
	\subfigure[reconstruct with 0.01$\%$ noise]{\includegraphics[width=0.3\textwidth]
		{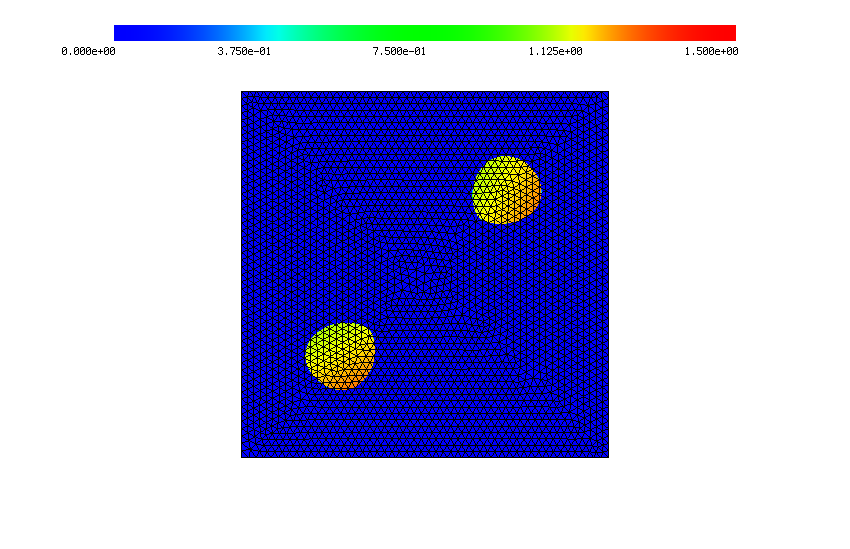}}	
	\subfigure[reconstruct with 0.1$\%$ noise]{\includegraphics[width=0.3\textwidth]
		{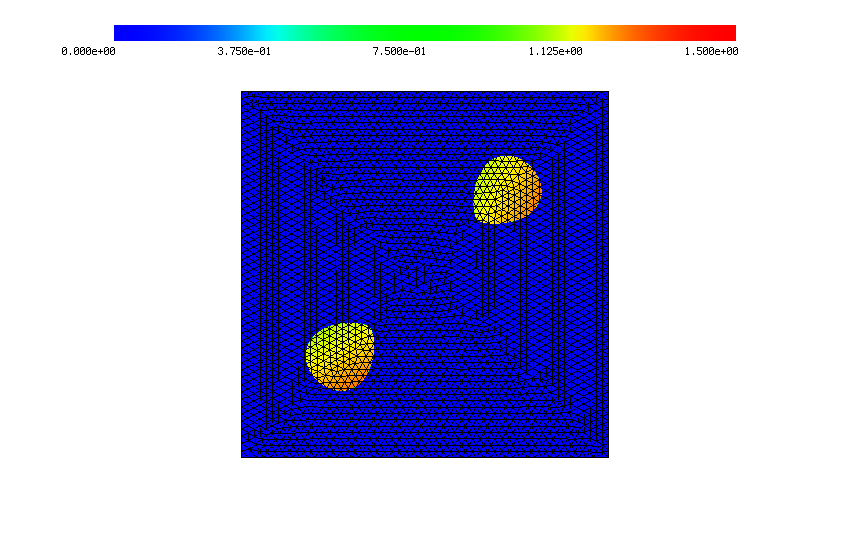}}
	\subfigure[reconstruct with 0.5$\%$ noise]{\includegraphics[width=0.3\textwidth]
		{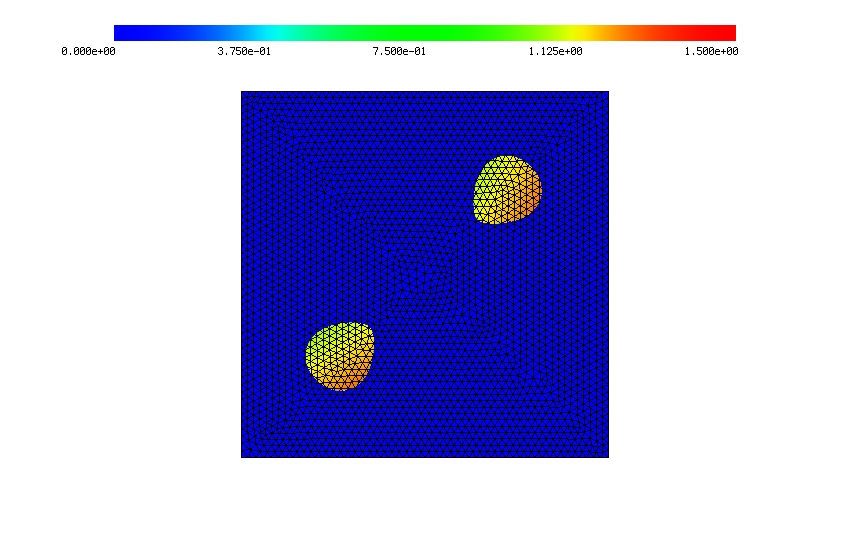}} \\
	\subfigure[reconstruct with 0.01$\%$ noise]
	{\includegraphics[width=0.3\textwidth]
		{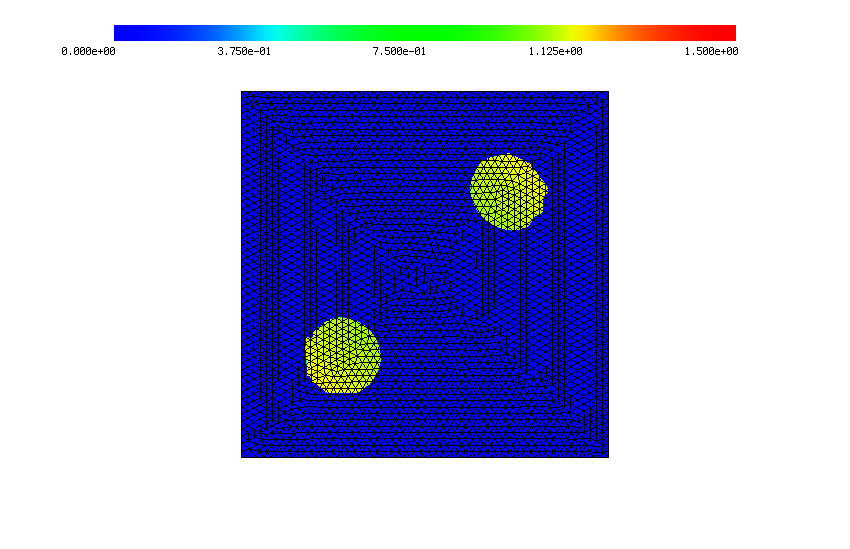}}	
	\subfigure[reconstruct with 0.1$\%$ noise]
	{\includegraphics[width=0.3\textwidth]
		{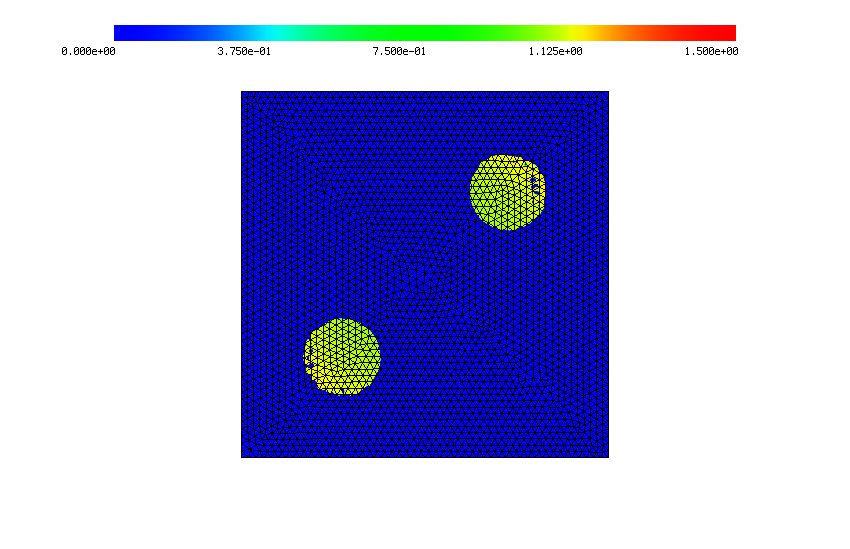}}
	\subfigure[reconstruct with 0.5$\%$ noise]
	{\includegraphics[width=0.3\textwidth]
		{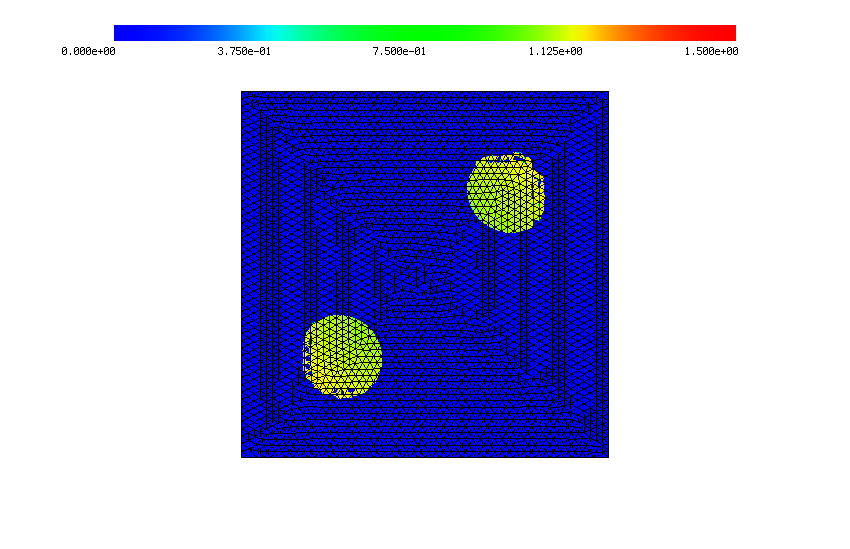}}
	\caption{Numerical results for Example \ref{ex:sepa} under different noise levels: (d)–(f) GZ method; (g)–(i) our method. The colorbars are uniformly set to [0, 1.5].}
	\label{fig:sepa}
\end{figure}

\begin{table}[htbp]
	\caption{Numerical results for Example \ref{ex:sepa} with different noise levels.}
	\centering
	\begin{tabular}{lcccc}%四个c代表该表一共四列，内容全部居中
		\hline
		Method & Error & $\delta=0.0001$ & $\delta=0.001$ & $\delta=0.005$ \\
		\hline
		GZ & ${\rm err} (\Omega_{\varepsilon})$ & $8.2751 \times 10^{-2}$ & $8.3527 \times 10^{-2}$ & $8.8435 \times 10^{-2}$ \\
		& ${\rm err} (\phi_{\varepsilon})$ & $1.4478 \times 10^{-1}$ & $1.4702 \times 10^{-1}$ & $1.5543 \times 10^{-1}$ \\			
		\hline
		Ours & ${\rm err} (\Omega_{\varepsilon})$ & $4.0068 \times 10^{-2}$ & $5.5411 \times 10^{-2}$ & $5.7244 \times 10^{-2}$ \\
		& ${\rm err} (\phi_{\varepsilon})$ & $5.7564 \times 10^{-2}$ & $5.7567 \times 10^{-2}$ & $5.7591 \times 10^{-2}$ \\			
		\hline
	\end{tabular}
	\label{tab:sepa}
\end{table}

Figure \ref{fig:sepa_near} further illustrates the capability of the proposed algorithm to recover multiple closely positioned sources. The exact domain is $\Omega_* = { (x,y) \in \Omega : (x \pm 0.2)^2 + (y \pm 0.2)^2 < 0.04 }$, with the same initial domain as before. The accurate recovery of multiple closely spaced sources is generally challenging. Figure~\ref{fig:sepa_near} (b) shows that our algorithm successfully evolved the initial single circular source into two very close independent true sources, with the CCBM approach playing a key role in stabilizing shape evolution.

\begin{figure}[htbp]
	\centering
	\subfigure[true source $\Omega_*$ ]
	{\includegraphics[width=0.4\textwidth]
		{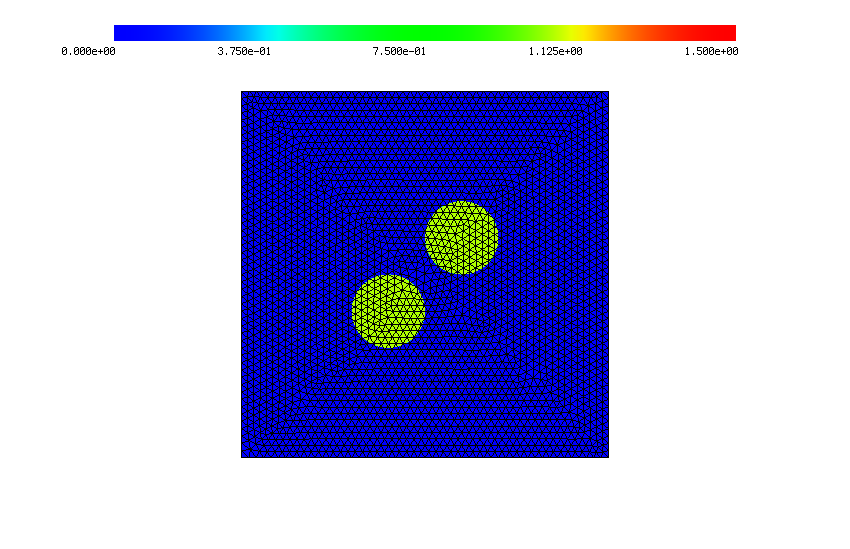}}
	\subfigure[reconstruct with 0.01$\%$ noise]
	{\includegraphics[width=0.4\textwidth]			
		{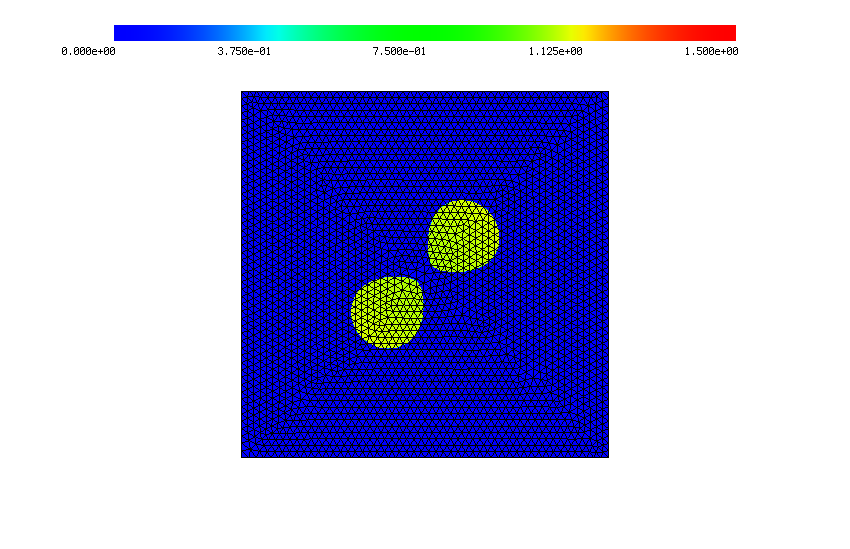}}		
	\caption{Numerical results for Example \ref{ex:sepa} with two near circles. The colorbars are uniformly set to [0, 1.5].}
	\label{fig:sepa_near}
\end{figure}

\begin{example} \label{ex:comb}
	In this example, we set $f=1 \ \textrm{in} \ \Omega$, with the exact source function $\phi_* = 2$ and the exact domain $\Omega_* = \{ (x,y) \in \Omega : x^2 + y^2 < 0.15 \}$. We choose the initial domain as $\Omega_{init} = \{ (x,y) \in \Omega : (x \pm 0.3)^2 + y^2 <0.02 \}$. 
\end{example}

In this example, we demonstrate the capability of the proposed algorithm to merge multiple connected domains, complementing the previous example in which the domain was split. When topological changes occur, the perimeter constraint becomes particularly important and must be carefully chosen. Examples~\ref{ex:sepa} and~\ref{ex:comb} together illustrate the effectiveness of the algorithm for handling topological changes in source inversion. 

The relative errors are listed in Table \ref{tab:comb}, and the reconstructed support and intensity are shown in Figure~\ref{fig:comb}. As observed in Figures \ref{fig:comb} (a) and (c), the numerical reconstructions closely approximate the exact solution, accurately recovering both the support and the intensity of the source. When uniformly distributed random noise with levels $\delta = 0.0001, 0.001, 0.005$ is added to boundary measurement $g_2$, the reconstructions in Figures \ref{fig:comb} (d)–(f) remain satisfactory. For comparison, Figures \ref{fig:comb} (g)–(i) present the results obtained by GZ. Under identical mesh resolution and noise conditions, our method achieves slightly more accurate reconstructions, and the relative errors exhibit limited sensitivity to imposed noise levels.

\begin{figure}[htbp]
	\centering
	\subfigure[true source $\Omega_*$ ]{\includegraphics[width=0.3\textwidth]
		{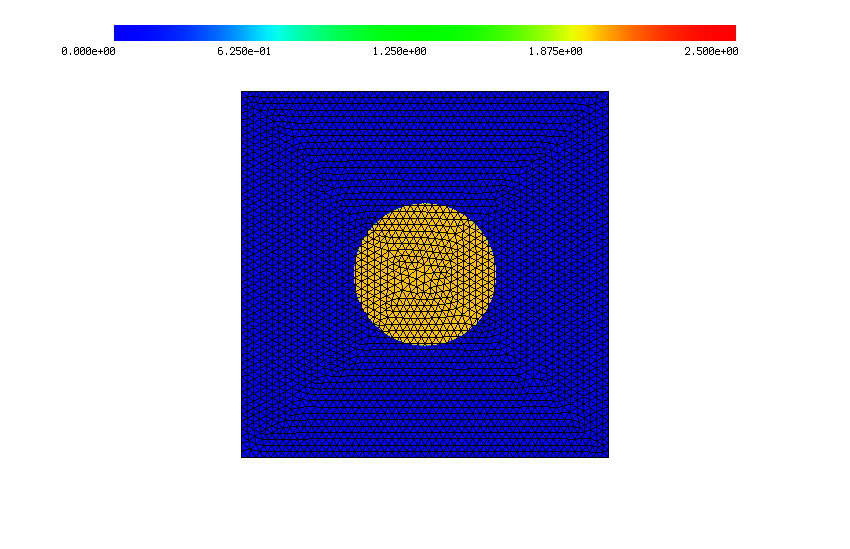}}
	\subfigure[initial domain $\Omega_{init}$]{\includegraphics[width=0.3\textwidth]
		{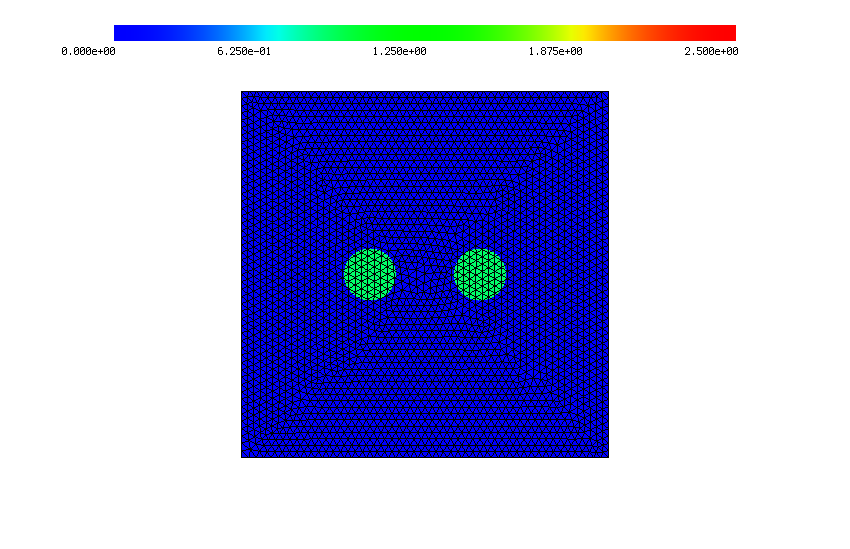}}
	\subfigure[reconstruct without noise]{\includegraphics[width=0.3\textwidth]
		{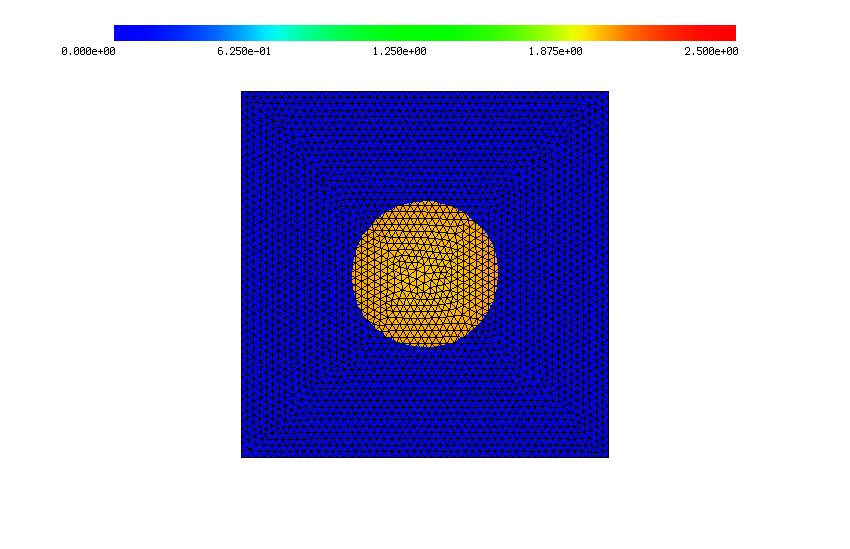}}\\
	\subfigure[reconstruct with 0.01$\%$ noise]{\includegraphics[width=0.3\textwidth]
		{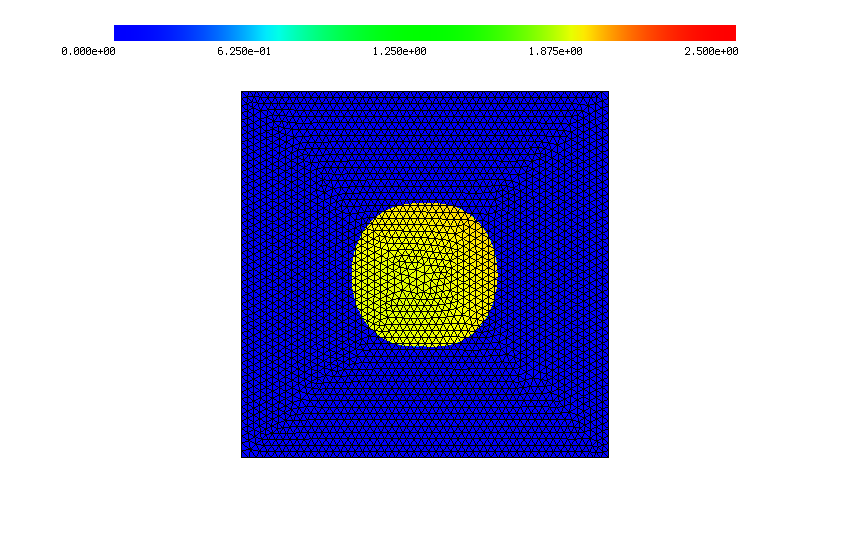}}	
	\subfigure[reconstruct with 0.1$\%$ noise]{\includegraphics[width=0.3\textwidth]
		{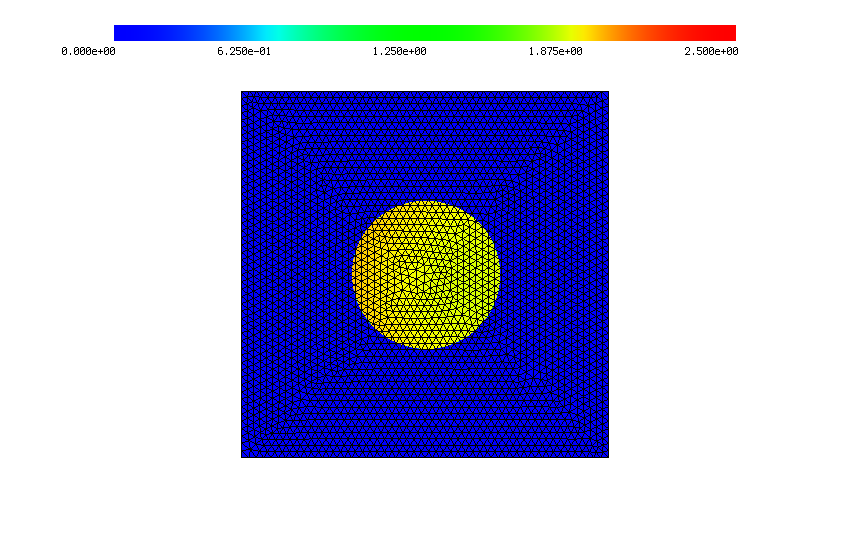}}
	\subfigure[reconstruct with 0.5$\%$ noise]{\includegraphics[width=0.3\textwidth]
		{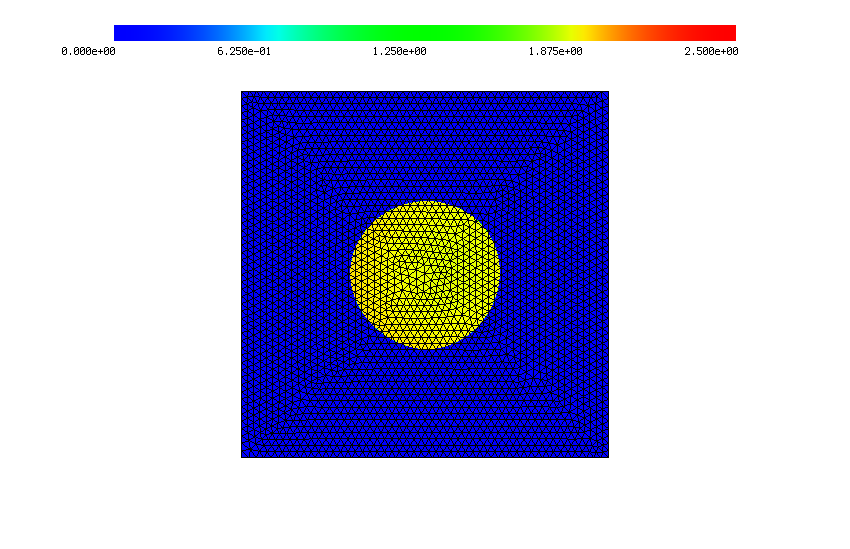}}\\
	\subfigure[reconstruct with 0.01$\%$ noise]{\includegraphics[width=0.3\textwidth]
		{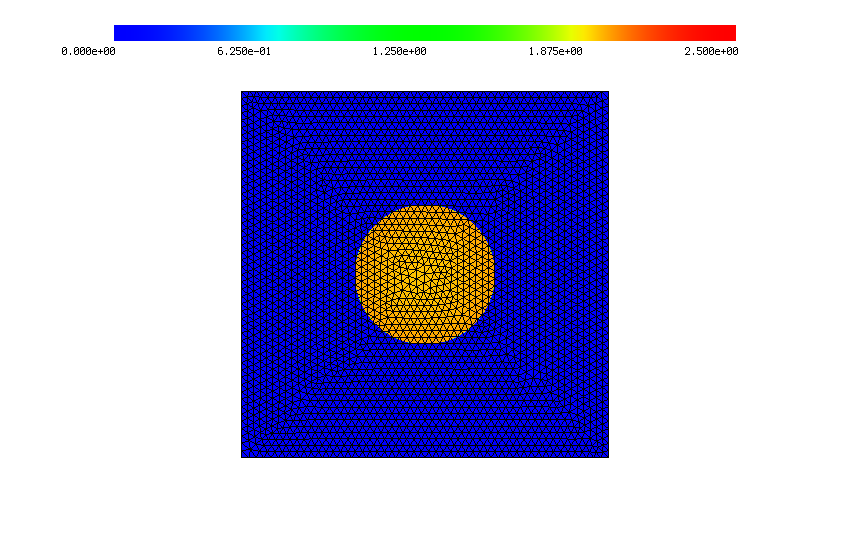}}	
	\subfigure[reconstruct with 0.1$\%$ noise]{\includegraphics[width=0.3\textwidth]
		{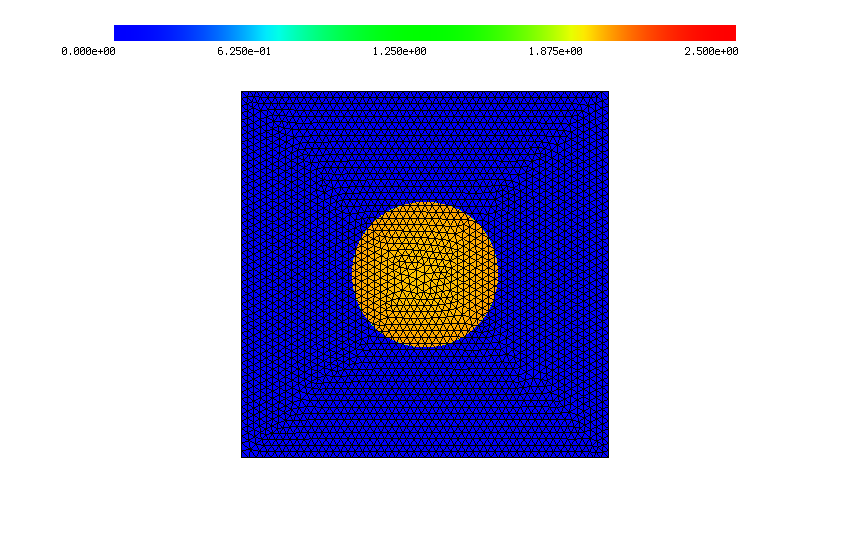}}
	\subfigure[reconstruct with 0.5$\%$ noise]{\includegraphics[width=0.3\textwidth]
		{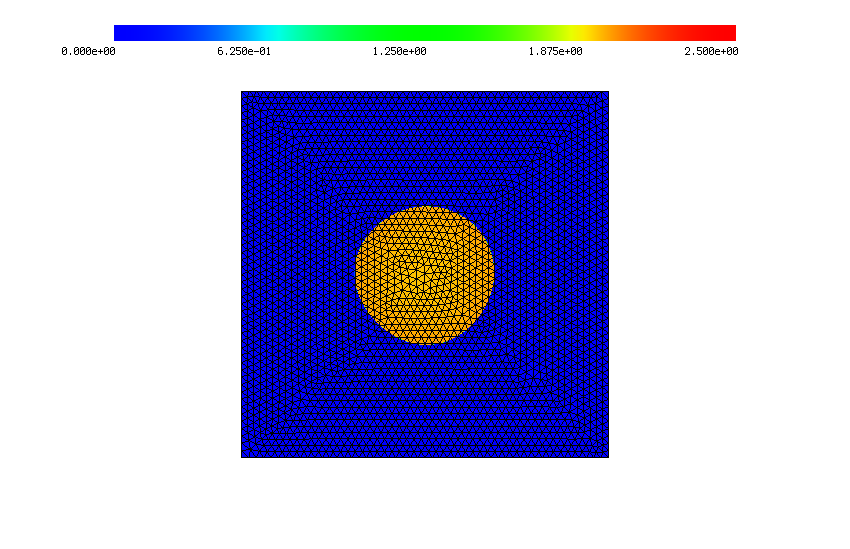}}
	\caption{Numerical results for Example \ref{ex:comb} under different noise levels: (d)–(f) GZ method; (g)–(i) our method. The colorbars are uniformly set to [0, 2.5].}
	\label{fig:comb}
\end{figure}

\begin{table}[htbp]
	\caption{Numerical results for Example \ref{ex:comb} with different noise levels.}
	\centering
	\begin{tabular}{lcccc}
		\hline
		Method & Error & $\delta=0.0001$ & $\delta=0.001$ & $\delta=0.005$ \\
		\hline
		GZ & ${\rm err} (\Omega_{\varepsilon})$ & $8.6546 \times 10^{-2}$ & $8.5601 \times 10^{-2}$ & $8.5910 \times 10^{-2}$ \\
		& ${\rm err} (\phi_{\varepsilon})$ & $4.3038 \times 10^{-2}$ & $4.3205 \times 10^{-2}$ & $4.3421 \times 10^{-2}$ \\			
		\hline 
		Ours & ${\rm err} (\Omega_{\varepsilon})$ & $5.4718 \times 10^{-2}$ & $5.8548 \times 10^{-2}$ & $6.2190 \times 10^{-2}$ \\
		& ${\rm err} (\phi_{\varepsilon})$ & $2.0824 \times 10^{-2}$ & $2.0826 \times 10^{-2}$ & $2.0827 \times 10^{-2}$
		\\			
		\hline
	\end{tabular}
	\label{tab:comb}
\end{table}

\begin{example}\label{ex:nest}
	In the final example, we set the exact source function $\phi_{*1}=5, \phi_{*2}=10$. The exact domain $\Omega_{*1}$ is defined as $\{(x,y) \in \Omega : 0.04 < x^2 + y^2 < 0.25\}$, $\Omega_{*2}$ is defined as $\{(x,y) \in \Omega : x^2 + y^2 < 0.04\}$. We choose the initial domain as $\Omega_{init} = \{(x,y) \in \Omega : x^2+y^2 < 0.07\}$. 
\end{example}

In this experiment, we consider the recovery of a piecewise constant source consisting of two nested layers. Reconstruction is performed in two stages. In the first stage, we employ the shape optimization procedure in Algorithm \ref{alg:SSDM} to identify the support of the outer layer, after which its intensity is obtained using the parameter-dependent CCBM scheme \eqref{eq:ccbm_intensity}. In the second stage, to isolate the contribution of the inner source, we subtract the boundary data generated by the recovered outer source from the total boundary measurements. The resulting residual data is used to reconstruct the support and intensity of the inner source within the same optimization framework.

The relative errors are listed in Table~\ref{tab:nest2}, and the reconstructed supports and intensities are shown in Figure \ref{fig:nest2}. As illustrated, the localization of the inner source exhibits a greater deviation than that of the outer source. This is primarily because the observation data for the inner source is affected by the reconstruction error of the outer source. To further substantiate this observation, we conducted an additional numerical experiment using a three-layer nested source.

\begin{figure}[htbp]
	\centering
	\subfigure[true source $\Omega_*$]
	{\includegraphics[width=0.3\textwidth]			
		{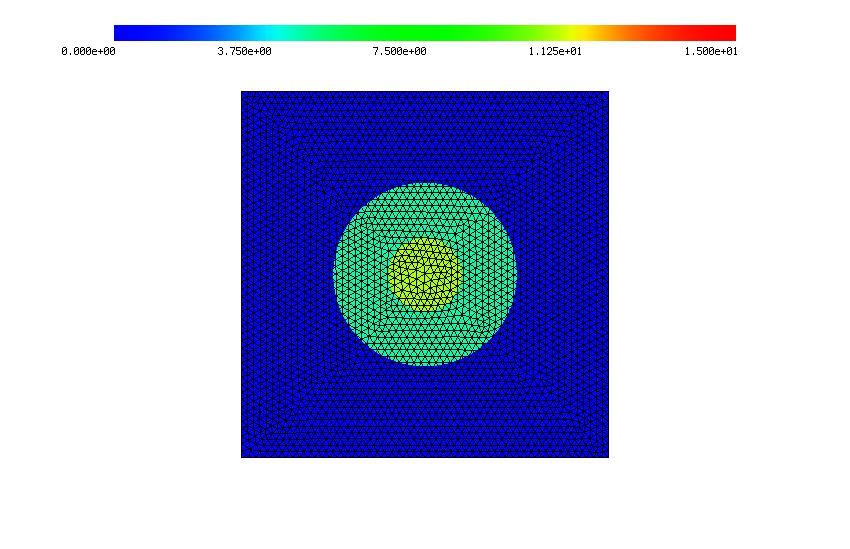}}
	\subfigure[initial domain $\Omega_{init}$]
	{\includegraphics[width=0.3\textwidth]			
		{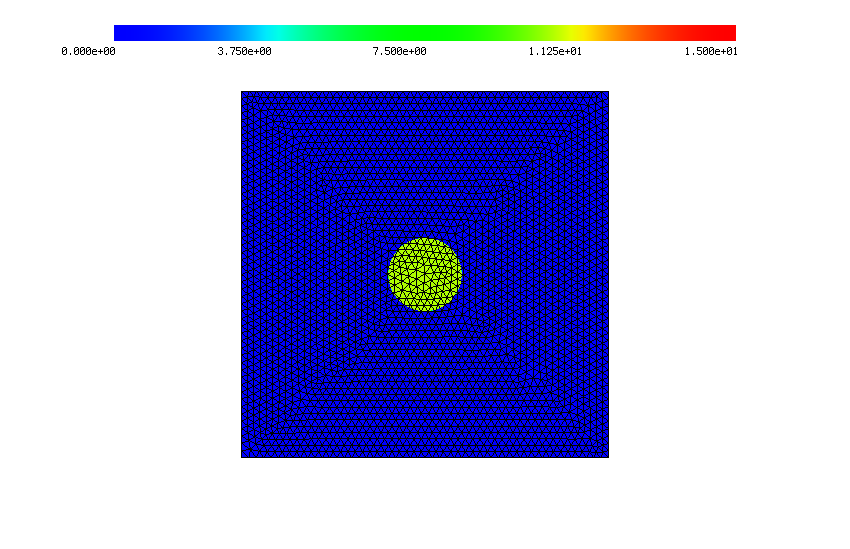}}
	\subfigure[reconstruct without noise]
	{\includegraphics[width=0.3\textwidth]			
		{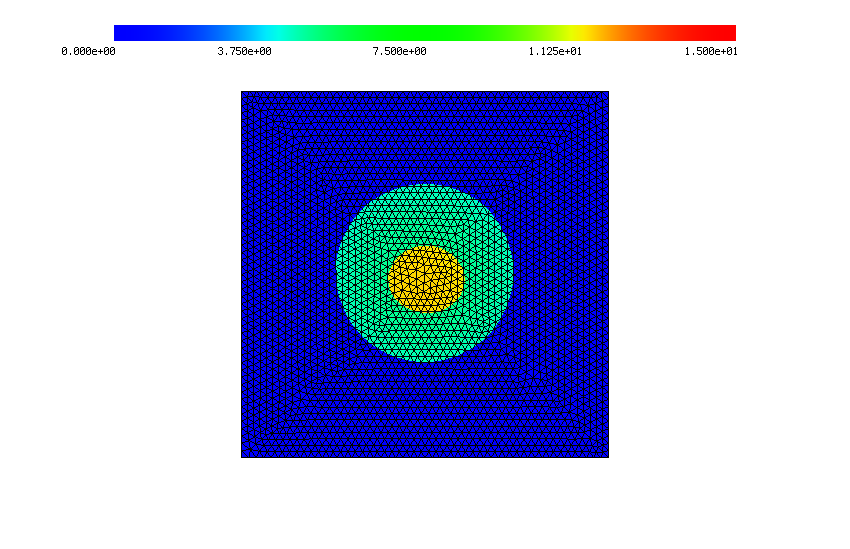}} \\
	\subfigure[reconstruct with 0.01$\%$ noise]
	{\includegraphics[width=0.3\textwidth]			
		{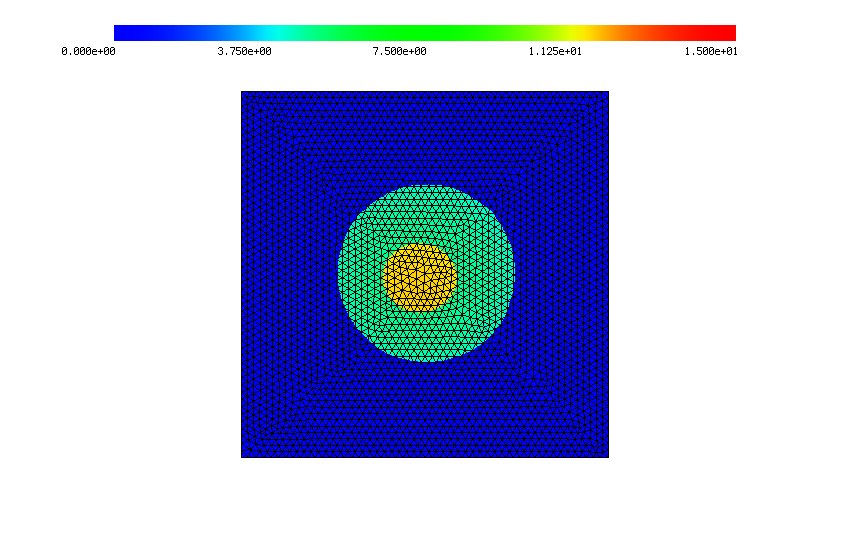}}	
	\subfigure[reconstruct with 0.1$\%$ noise]
	{\includegraphics[width=0.3\textwidth]			
		{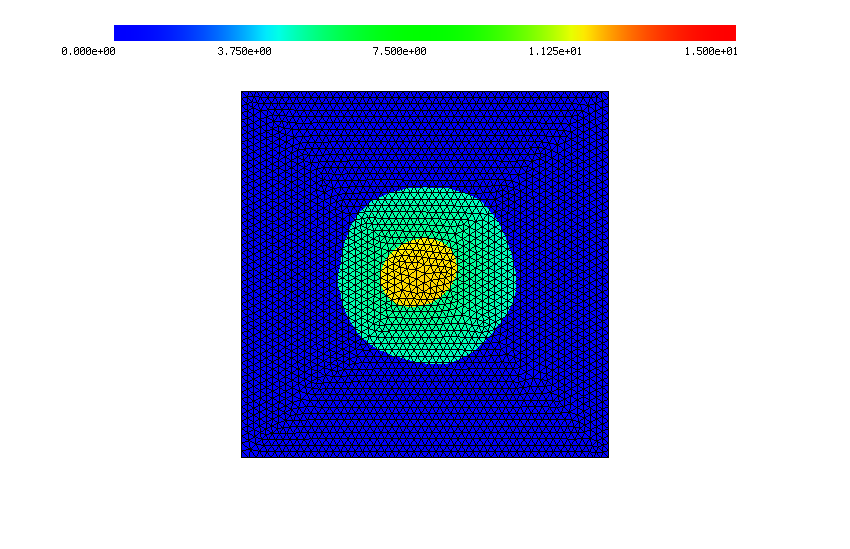}}
	\subfigure[reconstruct with 0.5$\%$ noise]
	{\includegraphics[width=0.3\textwidth]			
		{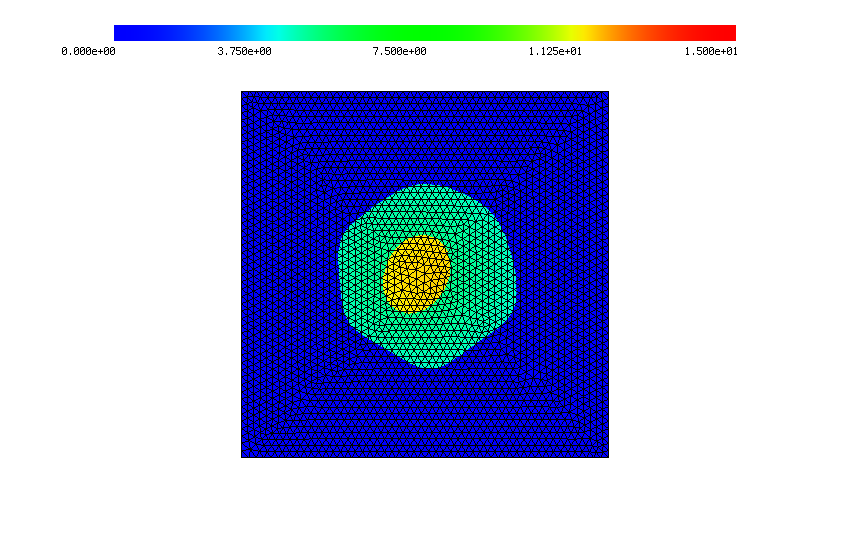}}	
	\caption{Numerical results for Example \ref{ex:nest} under different noise levels. The colorbars are uniformly set to [0, 15].}
	\label{fig:nest2}
\end{figure}

\begin{table}[htbp]
	\caption{Numerical results for Example \ref{ex:nest} with different noise levels.}
	\centering
	\begin{tabular}{lcccc}
		\hline
		Area & Error & $\delta=0.0001$ & $\delta=0.001$ & $\delta=0.005$ \\
		\hline
		$\Omega_{*1}$ & ${\rm err} (\Omega_{\varepsilon})$ & $3.1118 \times 10^{-2}$ & $3.2674 \times 10^{-2}$ & $6.0068 \times 10^{-2}$ \\
		& ${\rm err} (\phi_{\varepsilon})$ & $6.0123 \times 10^{-2}$ & $5.9684 \times 10^{-2}$ & $5.5568 \times 10^{-2}$ \\			
		\hline
		$\Omega_{*2}$ & ${\rm err} (\Omega_{\varepsilon})$ & $8.7105 \times 10^{-2}$ & $1.3717 \times 10^{-1}$ & $1.6037 \times 10^{-1}$ \\
		& ${\rm err} (\phi_{\varepsilon})$ & $5.7638 \times 10^{-2}$ & $5.9029 \times 10^{-2}$ & $5.9207 \times 10^{-2}$ \\	
		\hline
	\end{tabular}
	\label{tab:nest2}
\end{table}

We set the exact source function $\phi_{*1}=5, \phi_{*2}=10, \phi_{*3}=15$. The exact domain $\Omega_{*1}$ is defined as $\{(x,y) \in \Omega : 0.16 < x^2 + y^2 < 0.36\}$, $\Omega_{*2}$ is defined as $\{(x,y) \in \Omega : 0.04 < x^2 + y^2 < 0.16\}$, $\Omega_{*3}$ is defined as $\{(x,y) \in \Omega : x^2 + y^2 < 0.04\}$. We choose the initial domain as $\Omega_{init} = \{(x,y) \in \Omega : x^2+y^2 < 0.02\}$. 

The relative errors are summarized in Table~\ref{tab:nest3}, and the reconstructed supports and intensities are shown in Figure~\ref{fig:nest3}. As shown in Figures~\ref{fig:nest3} (a) and (b), the localization accuracy progressively deteriorates toward the inner layers, which consequently results in larger intensity errors.

\begin{figure}[htbp]
	\centering
	\subfigure[true source $\Omega_*$]
	{\includegraphics[width=0.4\textwidth]{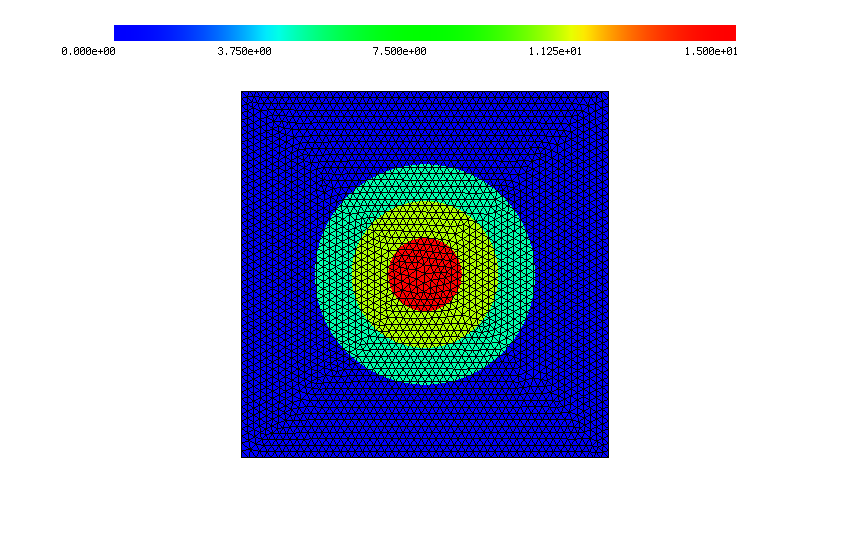}}
	\subfigure[reconstruct with 0.01$\%$ noise]
	{\includegraphics[width=0.4\textwidth]{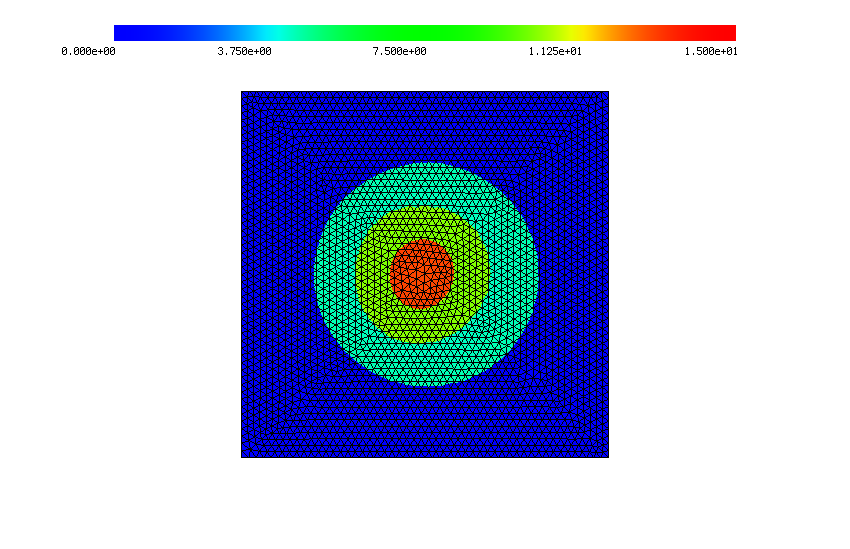}}
	\caption{Numerical results for Example \ref{ex:nest} with 3-layer nested source. The colorbars are uniformly set to [0, 15].}
	\label{fig:nest3}
\end{figure}

\begin{table}[htbp]
	\caption{Numerical results for Example \ref{ex:nest} with 3-layer nested source.}
	\centering
	\begin{tabular}{lccc}
		\hline
		Error & $\Omega_{*1}$ & $\Omega_{*2}$ & $\Omega_{*3}$ \\
		\hline
		${\rm err} (\Omega_{\varepsilon})$ & $2.1591 \times 10^{-2}$ & $7.5698 \times 10^{-2}$ & $1.2541 \times 10^{-1}$ \\
		${\rm err} (\phi_{\varepsilon})$ & $5.3790 \times 10^{-2}$ & $5.9983 \times 10^{-2}$ & $8.5075 \times 10^{-2}$ \\
		\hline
	\end{tabular}
	\label{tab:nest3}
\end{table}

\section{Conclusion} \label{sec:concl}
In this work, we address the inverse source problem in BLT with the goal of reconstructing both the support and intensity of internal light sources from boundary measurements. Our approach decouples intensity and geometry through a shape optimization framework and employs a level set representation, enabling the first truly nonparametric reconstruction of source regions without any a priori shape assumptions. Numerical experiments confirm the method’s ability to accurately recover multiple, closely spaced, or nested sources, even under noisy conditions. Specifically, for crescent-shaped and rectangular sources with sharp corners, where boundary features are difficult to capture, our method achieves satisfactory reconstruction of the source regions. For multiply connected regions, the reconstructed sources closely match the ground truth; even when a single initial source is used to reconstruct two closely spaced true sources, the boundaries are well separated. For nested sources, we observe that the reconstruction accuracy decreases for inner layers, likely due to the influence of outer layers on the corresponding boundary measurements. These results demonstrate clear advantages over existing methods in terms of robustness, accuracy, and flexibility, highlighting the method’s potential for practical BLT applications and its significance as a step forward in non-invasive source imaging.

%\data{All data that support the findings of this study are included within the article.}
%
%\ack{
%	The work of R. F. Gong is supported by the National Natural Science Foundation of China (No. 12071215, 12371302, 12571456). 	
%	The work of W. Gong is supported by the National Key Research and Development Program of China under grant 2022YFA1004402, and the National Natural Science Foundation of China (No. 12471393).
%	The work of S. F. Zhu is supported by the National Natural Science Foundation of China under grant (No.12471377), and the Science and Technology Commission of Shanghai Municipality (Nos. 23JC1400503 and 22DZ2229014).
%	}
	
%\section*{ORCID iDs}
%\noindent % 取消段落缩进
%Qianqian Wu \orcidlink{0009-0007-8349-3005}
%\href{https://orcid.org/0009-0007-8349-3005}{https://orcid.org/0009-0007-8349-3005} \\
%\noindent
%Rongfang Gong \orcidlink{0000-0001-5937-9607} \href{https://orcid.org/0000-0001-5937-9607}{https://orcid.org/0000-0001-5937-9607} \\
%\noindent
%Wei Gong \orcidlink{0000-0003-3251-8771}
%\href{https://orcid.org/0000-0003-3251-8771}{https://orcid.org/0000-0003-3251-8771} \\
%\noindent
%Ziyi Zhang \orcidlink{0009-0004-2800-7551}
%\href{https://orcid.org/0009-0004-2800-7551}{https://orcid.org/0009-0004-2800-7551} \\
%\noindent
%Shengfeng Zhu \orcidlink{0000-0002-4626-9983}
%\href{https://orcid.org/0000-0002-4626-9983}{https://orcid.org/0000-0002-4626-9983}

\bibliographystyle{unsrt}  % 或者其他样式: abbrv, ieeetr, unsrt
\bibliography{references}

%\section*{References}

%
% Each of the commands below will create an unnumbered section with the appropriate heading.
% Remove any sections that are not relevant for your article.
% All sections except suppdata will be removed if the [anonymous] option is used.
% See iopjournal-guidelines.pdf for more information.
%

%\funding{Sample text inserted for demonstration.}
%% This section is a list of funder names and grant numbers
%
%\roles{Sample text inserted for demonstration.}
%% List author names and the contributions made to the article, using terms from the NISO Contributor Roles Taxonomy (CRediT) https://credit.niso.org
%
%
%\suppdata{Sample text inserted for demonstration.}

\end{document}